\documentclass[psamsfonts]{amsart}
\usepackage{amssymb}
\usepackage{graphicx,color}
\usepackage{amssymb,amsfonts}
\usepackage[all,arc]{xy}
\usepackage{enumerate}
\usepackage{mathrsfs}

\newtheorem{thm}{Theorem}[section]
\newtheorem{lemma}[thm]{Lemma}
\newtheorem{prop}[thm]{Proposition}
\newtheorem{cor}[thm]{Corollary}

\theoremstyle{definition}
\newtheorem{defn}[thm]{Definition}
\newtheorem{exmp}[thm]{Example}

\theoremstyle{remark}
\newtheorem{rem}[thm]{Remark}
\makeatletter
\let\c@equation\c@thm
\makeatother
\numberwithin{equation}{section}

\title[Hypersurfaces of Prescribed Curvature ]{Lipschitz Continuous Hypersurfaces with Prescribed Curvature and Asymptotic Boundary in Hyperbolic Space}

\author{Zhenan Sui}

\address{Institute for Advanced Study in Mathematics of HIT, Harbin Institute of Technology, Harbin, China}
\email{sui.4@osu.edu}

\author{Wei Sun}

\address{Institute of Mathematical Sciences, ShanghaiTech University, Shanghai, China}
\email{sunwei@shanghaitech.edu.cn}

\begin{document}

\begin{abstract}
We prove the existence of a complete locally Lipschitz continuous hypersurface in weak sense with prescribed Weingarten curvature and asymptotic boundary at infinity in hyperbolic space under certain assumptions.
\end{abstract}


\maketitle


\section {\large Introduction}

\vspace{4mm}

This paper is devoted to the study of asymptotic Plateau type problem in hyperbolic space, for which,
we shall use the half space model
\[\mathbb{H}^{n+1} = \{ (x, x_{n+1}) \in \mathbb{R}^{n+1} \big\vert x_{n+1} > 0\}\]
with the metric
\[ d s^2 = x_{n+1}^{- 2} \sum_{i = 1}^{n+1} d x_i^2. \]
Given a smooth positive function $\psi$ in $\mathbb{H}^{n + 1}$ and a disjoint collection of smooth closed $(n - 1)$ dimensional submanifolds $\Gamma = \{\Gamma_1, \ldots, \Gamma_m\}$ at $\partial_{\infty} \mathbb{H}^{n+1} = \mathbb{R}^n \times \{0\}$, we want to find a complete connected admissible vertical graph $\Sigma = \{ (x, u(x) ) | x \in \Omega\}$ satisfying
\begin{equation} \label{eqn9}
\left\{ \begin{aligned}
f ( \kappa [ u ] ) =  & \,\sigma_k^{\frac{1}{k}} ( \kappa ) = \psi(x, u) \quad & \mbox{in} \,\, \Omega, \\
u  = & \,  0 \quad & \mbox{on} \,\, \Gamma,
\end{aligned} \right.
\end{equation}
where $\kappa = (\kappa_1, \ldots, \kappa_n)$ are the hyperbolic principal curvatures of $\Sigma$ with respect to the upward normal, the $k$th-Weingarten curvature
\[\sigma_k (\kappa) = \sum_{1 \leq i_1 < \ldots < i_k \leq n} \kappa_{i_1} \cdots \kappa_{i_k} \]
is defined on $k$-th G\r arding's cone
\[\Gamma_k \equiv \{ \kappa \in \mathbb{R}^n \vert \sigma_j(\kappa) > 0,\, j = 1, \ldots, k \},\]
and $\Omega$ is the bounded domain enclosed by $\Gamma$ on $\mathbb{R}^n \times \{ 0 \}$.
We say $\Sigma$ is admissible if $\kappa \in \Gamma_k$.

The difficulty for Plateau type problem \eqref{eqn9} lies in the singularity at $\Gamma$. A common method to deal with such problem is by studying approximating Dirichlet problem
\begin{equation} \label{eqn11}
\left\{ \begin{aligned} f ( \kappa [ u ] ) =  & \, \psi(x, u) \quad & \mbox{in} \,\, \Omega, \\
u   = & \, \epsilon \quad & \mbox{on} \,\, \Gamma,
\end{aligned} \right.
\end{equation}
where $\epsilon$ is a small positive constant. When $\psi = \sigma \in (0, 1)$ is a prescribed constant and $f$ satisfies certain assumptions, extensive study can be found in \cite{GS00, GSS09, GS10, GS11, GSX14}, where the estimates for solutions to \eqref{eqn11} have to be $\epsilon$-independent in order to prove existence results for asymptotic problem \eqref{eqn9}.
For nonconstant $\psi$,  Szapiel \cite{Sz05} investigated the existence of strictly locally convex solutions to the approximating problem \eqref{eqn11}.

In \cite{Sui2019}, the author constructed a new approximating Dirichlet problem by assuming the existence of a strictly locally convex asymptotic subsolution. Combined with interior estimates for the case $k = 2$, existence results can be concluded for strictly locally convex solutions to asymptotic problem \eqref{eqn9}, even when the estimates for the approximating problem depend on $\epsilon$. In this paper, we shall continue to adopt this idea to find admissible hypersurfaces.

Assume that there exists an admissible $\underline{u} \in C^4(\Omega) \cap C^0(\overline{\Omega})$ such that
\begin{equation} \label{eqn8}
\left\{ \begin{aligned}
f(\kappa [ \underline{u} ]) \geq  & \, \psi(x, \underline{u}) \quad & \mbox{in} \,\, \Omega, \\
\underline{u}  = & \, 0 \quad  & \mbox{on} \,\, \Gamma.
\end{aligned} \right.
\end{equation}
Denote the $\epsilon$-level set of $\underline{u}$ and its enclosed region in $\mathbb{R}^n$ by
\[ \Gamma_{\epsilon} =  \{ x \in \Omega \,\big\vert\, \underline{u}(x)  = \epsilon\}, \quad \quad
\Omega_{\epsilon} = \{ x \in \Omega \,\big\vert\, \underline{u}(x)  > \epsilon\}. \]
We assume that $\Gamma_{\epsilon}$ is a regular boundary of $\Omega_{\epsilon}$ when $\epsilon > 0$ is sufficiently small. That is to say, $\Gamma_{\epsilon}$ has dimension $n - 1$, $\Gamma_{\epsilon} \in C^4$ and $\underline{u}_{\gamma} = |D \underline{u}| > 0$ on $\Gamma_{\epsilon}$, where $\gamma$ is the unit interior normal vector field to $\Gamma_{\epsilon}$ on $\Omega_{\epsilon}$. Here the requirement for $\underline{u}$ to be $C^4$ is for second order boundary estimate.
Throughout this paper, we shall consider the following approximating Dirichlet problem
\begin{equation} \label{eqn10}
\left\{\begin{aligned}
f ( \kappa [ u ] ) =  & \, \psi(x, u) \quad & \mbox{in} \,\, \Omega_{\epsilon}, \\
u  = & \, \epsilon \quad & \mbox{on} \,\, \Gamma_{\epsilon}.
\end{aligned} \right.
\end{equation}

Before we state our main theorems, let us first impose some compatibility conditions, which are needed for boundary gradient estimate on $\Gamma_{\epsilon}$.
For any $\epsilon > 0$ sufficiently small, let $\sigma \in (0, 1)$ be a constant which satisfies
\begin{equation} \label{eqn19}
\psi(x, u)  >  \sigma_k^{\frac{1}{k}} ( \sigma, \ldots, \sigma )  \quad\quad\mbox{on} \quad \overline{\Omega_{\epsilon}}
\end{equation}
for any admissible solution $u \geq \underline{u}$ to \eqref{eqn10}.
Note that such $\sigma$ exists in view of Remark \ref{Remark1}, and $\sigma$ may depend on $\epsilon$. Denote by $r_0^{\epsilon}$ the maximal radius of exterior spheres to $\Gamma_{\epsilon}$ in $\mathbb{R}^n$.
We impose the following compatibility conditions for \eqref{eqn10} and $\underline{u}$:
\begin{equation} \label{eq1-4}
0 < \epsilon < r_0^{\epsilon} \sigma \quad \mbox{and} \quad
\sigma - \frac{\sqrt{1 - \sigma^2}}{r_0^{\epsilon}} \,\epsilon - \frac{1 + \sigma}{(r_0^{\epsilon})^2}\,\epsilon^2 > 0.
\end{equation}
We note that the compatibility conditions are mild and can embrace the case when $\psi$ approaches $0$ on $\Gamma$, at which problem \eqref{eqn9} becomes both singular and degenerate.

Our first result is on the existence of admissible hypersurfaces to approximating problem \eqref{eqn10}.
\begin{thm} \label{Theorem1}
Suppose that $0 < \psi(x, u) \in C^{\infty} (\mathbb{H}^{n + 1})$ satisfies
\begin{equation} \label{eq1-2}
\psi_u - \frac{\psi}{u} \geq 0,
\end{equation}
and there exists an admissible $\underline{u} \in C^4(\Omega) \cap C^0(\overline{\Omega})$ satisfying \eqref{eqn8} and
\begin{equation} \label{eq1-3}
- \lambda (D^2 \underline{u}) \in \Gamma_{k + 1}  \quad  \mbox{near} \,\, \Gamma.
\end{equation}
For sufficiently small constant $\epsilon > 0$, assume that the compatibility conditions \eqref{eq1-4} hold for \eqref{eq3-12}, \eqref{eq3-13} and $\underline{u}$. Then there exists a unique smooth admissible solution $u^{\epsilon} \geq \underline{u}$ to the approximating problem \eqref{eqn10} in $\overline{\Omega_{\epsilon}}$. When $k = n$, condition \eqref{eq1-2}, \eqref{eq1-3} and \eqref{eq1-4} can be removed and our conclusion remains true except the uniqueness.
\end{thm}
The proof of Theorem \ref{Theorem1} relies on the establishment of a priori second order estimates for admissible solutions $u \geq \underline{u}$ of \eqref{eqn10}. These estimates depend on $\epsilon$, which is the minimum of $u$ on $\overline{\Omega_{\epsilon}}$, and thus we are able to apply techniques for usual Dirichlet problems. For the special case $k = n$, we refer the readers to \cite{Sui2019}, where the estimates can be derived by property of strict local convexity. For general $k$, we adopt the idea of Guan-Spruck \cite{GS10} to derive $C^0$ estimate and boundary gradient estimate, which rely directly and inevitably on the geometry of hyperbolic space.
For global curvature estimate, we construct a test function making use of geometric quantities in half space model, which, easily brings us to derivations similar to Jin-Li \cite{JL05}, where they used spherical coordinates to find starshaped compact radial graphs.

For second order boundary estimate, we shall generalize the idea of Ivochkina, Lin and Trudinger \cite{Ivochkina89, Lin-Trudinger94} to hyperbolic space.
For barrier construction, following \cite{Ivochkina89, Lin-Trudinger94}, we need
to guarantee that the principal curvatures $\kappa' = (\kappa'_1, \ldots, \kappa'_{n - 1})$ of $\Gamma_{\epsilon}$ with respect to $\gamma$ satisfy
$\kappa' \in \Gamma'_k$ on $\Gamma_{\epsilon}$, where
\[ \Gamma'_{k} \equiv \{ \kappa' \in \mathbb{R}^{n - 1} \vert \sigma_j(\kappa') > 0,\, j = 1, \ldots, k \}. \]
By the relation
\[ D_{\alpha \beta} \underline{u} = - \underline{u}_{\gamma} \kappa'_{\alpha} \delta_{\alpha \beta} \quad \text{ on } \Gamma_{\epsilon}, \]
where $\alpha$, $\beta < n$ run over the principal directions on $\Gamma_{\epsilon}$,
we observe that $\lambda( D_{\alpha \beta} \underline{u})$ and $\kappa'$ differ by a negative sign, which is very different from the Euclidean space where we would have $\underline{u}_{\gamma} \leq 0$. Therefore, we add condition \eqref{eq1-3}. As a result, we find two big differences due to the hyperbolic space and our problem setting. First, condition \eqref{eq1-3} can not be weakened to $- \lambda (D^2 \underline{u}) \in \Gamma_{k}$ as in Euclidean space \cite{Ivoch-Lin-Tru96}. Second, our second order boundary estimate depends on $\inf_{\Omega_{\epsilon}} \psi$, while in the Euclidean space \cite{Lin-Trudinger94}, it can be independent.

To solve the asymptotic problem \eqref{eqn9}, we utilize the interior gradient estimate to give a $\epsilon$-independent $C^1$ bound for solution sequence  $u^{\epsilon}$ of \eqref{eqn10} with $0 < \epsilon < \frac{\epsilon_0}{2}$ on fixed $\overline{\Omega_{\epsilon_0}}$. By diagonal process, we can then prove the existence of a locally Lipschitz continuous hypersurface to \eqref{eqn9} in weak sense.
At this point, we mention that condition \eqref{eq1-2} is indispensable in $\mathbb{H}^{n + 1}$ when $k < n$. It is needed for both global gradient estimate to \eqref{eqn10} and interior gradient estimate (see Weng \cite{Weng2019}).
In $\mathbb{R}^{n + 1}$, condition \eqref{eq1-2} becomes $\psi_u \geq 0$, which was used by Wang \cite{WangXJ98} to obtain interior gradient estimate. Before we state the existence theorem to problem \eqref{eqn9}, we give an example of what our data can be.
\begin{exmp}
Let
\[\Omega = \big\{ x \in \mathbb{R}^n \big\vert \, |x| < (1 - \sigma_1^2)^{\frac{1}{2}}  R \big\}, \]
where $\sigma_1 \in (0, 1)$ and $R > 0$ are constants. Let
$\psi = \alpha u^2$, where
\[ \alpha = \frac{\sigma_k^{\frac{1}{k}} (\sigma_1, \ldots, \sigma_1)}{(1 - \sigma_1)^2 R^2}. \]
Choose $\underline{u} = \sqrt{R^2 - |x|^2} - \sigma_1 R$. It will become clear in section 2 that
$\kappa[\underline{u}] = (\sigma_1, \ldots, \sigma_1)$. For any $0 < \epsilon < (1 - \sigma_1) R$, we may pick
\[ \sigma  =  \frac{\sigma_1 \epsilon^2}{2 (1 - \sigma_1)^2 R^2}. \]
Note that $r_0^{\epsilon} = \infty$. We can verify that all the assumptions in Theorem \ref{Theorem1} are satisfied.
\end{exmp}

\begin{thm} \label{Theorem2}
Under the assumptions of Theorem \ref{Theorem1}, there exists a sequence of admissible solution $u^{\epsilon} \in C^{\infty}(\overline{\Omega_{\epsilon}})$ to \eqref{eqn10} such that $\epsilon \searrow 0$ and $u^{\epsilon}$ converges to $u \in C^{0, 1}_{loc}(\Omega) \cap C^0(\overline{\Omega})$. We call $u$ weak admissible solution to asymptotic Plateau problem \eqref{eqn9}. When $k = n$, condition \eqref{eq1-2}, \eqref{eq1-3} and \eqref{eq1-4} can be removed and our conclusion remains true. When $k = n = 2$ or $k = 1$, $u$ can be smooth.
\end{thm}
Our definition of weak admissible solution may be interpreted in the spirit of Trudinger \cite{Trudinger97}, which was originally defined for Hessian equations. In Section 7, we shall prove that $u$ is indeed a viscosity solution to \eqref{eqn9}, as defined in \cite{Trudinger90}. In \cite{Sui2019}, the author applied Guan-Qiu's idea \cite{GQ17} to derive interior $C^2$ estimate for strictly locally convex solutions when $k = 2$. As a result, smooth solution to asymptotic problem \eqref{eqn9} can be found in the special case $k = n = 2$. However, interior $C^2$ estimate cannot be derived for higher order Weingarten curvature equations ($k \geq 3$) in view of the counterexamples given by Pogorelov \cite{Po78} and Urbas \cite{Ur90}. Thus, in Section 6, we formulate some possible domains on which we wish to establish Pogorelov type interior curvature estimate, but then we find an obstruction due to the hyperbolic space. Therefore, we wish to answer the following questions in future work: whether there exists a non-smooth locally Lipschitz continuous viscosity solution to \eqref{eqn9} when there is an asymptotic subsolution. If so, what is the optimal regularity of our weak admissible solution $u$.

The rest of this paper is organized as follows: the proof of Theorem \ref{Theorem1} is covered in Section 2--5. Combined with interior gradient estimate in Section 6, we finish the proof of Theorem \ref{Theorem2}.

\vspace{4mm}

\section{ \large $C^1$ estimate}

\vspace{4mm}

First, we shall present some preliminary knowledge which may be found in \cite{GSS09, GS10, GS11, GSX14, Sui2019}. The coordinate vector fields on vertical graph of $u$ are given by
\[  \partial_i + u_i \partial_{n + 1}, \quad i = 1, \ldots, n, \]
where $\partial_{i} = \frac{\partial}{\partial x_{i}}$ for $i = 1, \ldots, n + 1$ are the coordinate vector fields in $\mathbb{R}^{n+1}$.

When $\Sigma = \{ (x, u(x)) | x \in \Omega \}$ is viewed as a hypersurface in $\mathbb{R}^{n + 1}$, its upward unit normal,  metric, inverse of the metric and second fundamental form are respectively
\[\nu = \frac{1}{w} ( - D u, 1 ), \quad w = \sqrt{ 1 + |D u |^2},  \]
\[ \tilde{g}_{ij} = \delta_{ij} + u_i u_j, \quad  \tilde{g}^{ij} =  \delta_{ij} - \frac{u_i u_j}{w^2},  \quad  \tilde{h}_{ij} = \frac{u_{ij}}{w}. \]
The Euclidean principal curvatures $\tilde{\kappa}$ are the eigenvalues of the symmetric matrix
\[ \tilde{a}_{ij} = \frac{1}{w} \gamma^{ik} u_{kl} \gamma^{lj} \quad \mbox{with} \,\,
 \gamma^{ik} = \delta_{ik} - \frac{u_i u_k}{w ( 1 + w )}, \quad
 \gamma_{ik} = \delta_{ik} + \frac{u_i u_k}{1 + w}. \]
Note that $\gamma^{ik} \gamma_{kj} = \delta_{ij}$ and $\gamma_{ik} \gamma_{kj} = \tilde{g}_{ij}$.

When $\Sigma = \{ (x, u(x)) | x \in \Omega \}$ is viewed as a hypersurface in $\mathbb{H}^{n + 1}$, its unit upward normal, metric, second fundamental form are given as follows
\[ {\bf n} = u \nu, \quad  g_{ij} = \frac{1}{u^2} ( \delta_{ij} + u_i u_j ), \quad
 h_{ij} = \frac{1}{u^2 w} ( \delta_{ij} + u_i u_j + u u_{ij} ). \]
The hyperbolic principal curvatures $\kappa [u]$ are the eigenvalues of the symmetric matrix $A [u] = \{ a_{ij} \}$, whose entries are given by
\[ a_{ij} =  u^2 \gamma^{ik} h_{kl} \gamma^{lj}
       = \frac{1}{w} \gamma^{ik} ( \delta_{kl} + u_k u_l + u u_{kl} ) \gamma^{lj} =  \frac{1}{w} ( \delta_{ij} + u \gamma^{ik} u_{kl} \gamma^{lj} ). \]
Equation \eqref{eqn9} can be written as
\begin{equation} \label{eq1-1}
  f( \kappa [ u ] ) = f( \lambda( A[ u ] )) = F( A[ u ] ) =  \psi(x, u).
\end{equation}

From the above discussion, we obtain the following relations.
\begin{equation} \label{eq0-3}
h_{ij} = \frac{1}{u} \tilde{h}_{ij} + \frac{\nu^{n+1}}{u^2} \tilde{g}_{ij},
\end{equation}
where $\nu^{n+1} = \nu \cdot \partial_{n + 1}$ and $\cdot$ is the inner product in $\mathbb{R}^{n+1}$. Note that this formula indeed holds for any local frame on any hypersurface $\Sigma$
which may not be a graph. In addition, we have
\begin{equation} \label{eqn14}
\kappa_i =  u \tilde{\kappa_i} + \nu^{n+1},  \quad i = 1, \ldots, n.
\end{equation}

In the rest of this section and section 3, 4, we will establish $C^2$ a priori estimate for admissible solutions $u \geq \underline{u}$ to approximating problem \eqref{eqn10}. Our estimate will depend on $\epsilon$.

We shall need the following type of maximum principle in hyperbolic space, which originally appears in \cite{Sz05}.

\begin{lemma} \label{LemmaA}
Let $\Omega' \subset \Omega$ be a domain and $u$, $v$ be positive $C^2$ functions on $\Omega'$, where $u$ is admissible and $\kappa[v] \in \overline{\Gamma_k}$. Assume that $f(\kappa[ v ]) < f(\kappa[ u ])$ in $\Omega'$. If $u - v$ has a local maximum at $x_0 \in \Omega'$, then $u(x_0) \neq v(x_0)$.
\end{lemma}

\begin{proof}
Prove by contradiction. Suppose that $u(x_0) = v(x_0)$. By assumption we know that $D u (x_0) = D v (x_0)$ and $D^2 u(x_0) \leq D^2 v(x_0)$. Therefore at $x_0$,
\[ A [ u ] = \frac{1}{w} \big( \delta_{ij} + u \gamma^{ik} u_{kl} \gamma^{lj} \big) \leq  \frac{1}{w} \big( \delta_{ij} + v \gamma^{ik} v_{kl} \gamma^{lj} \big) = A[ v ]. \]
Consequently, $f(\kappa[ u ]) ( x_0 ) \leq f(\kappa[ v ]) ( x_0 )$. This is a contradiction.
\end{proof}

\subsection{$C^0$ estimate}

\vspace{2mm}

For $\sigma \in [0, 1)$, let $B^{\sigma} = B_R^{\sigma} =  B^{\sigma}_R(a)$ be a ball in $\mathbb{R}^{n+1}$ of radius $R$ centered at $a = (a', \,- \sigma R)$ and $S^{\sigma} = S_R^{\sigma} = \partial B_R^{\sigma} \cap \mathbb{R}^{n+1}_+$. By \eqref{eqn14}, we know that $\kappa_i[S^{\sigma}] = \sigma$ for all $i$ with respect to its outward normal.

\begin{lemma} \label{lemma5-1}
There exists a ball $B^{\sigma}_{R}(a)$ such that for any admissible solution $u \geq \underline{u}$ to \eqref{eqn10}, the graph $\Sigma^{\epsilon} = \{ (x, u(x)) | x \in \Omega_{\epsilon} \}$ is contained in $B^{\sigma}_{R}(a)$.
\end{lemma}
\begin{proof}
Let $\sigma \in [0, 1)$ be a constant satisfying \eqref{eqn19}.
Since $\Gamma_{\epsilon} \times \{ \epsilon \}$ is compact, we can choose a ball $B^{\sigma}_R (a)$ such that $\Gamma_{\epsilon} \times \{ \epsilon \} \subset B^{\sigma}_R(a)$. Let $\Sigma^{\epsilon}$ be an admissible hypersurface to \eqref{eqn10}. Suppose $\Sigma^{\epsilon}$ is not contained in $B^{\sigma}_R(a)$. Expand $B^{\sigma}$ continuously by homothetic dilation from $(a', 0)$ until $B^{\sigma}$ contains $\Sigma^{\epsilon}$ and then reverse the procedure until $S^{\sigma}$ has a first contact with $\Sigma^{\epsilon}$. However,  $S^{\sigma}$ and $\Sigma^{\epsilon}$ can not have a first contact by Lemma \ref{LemmaA}. Hence $\Sigma^{\epsilon} \subset B^{\sigma}_R (a)$.
\end{proof}

\begin{rem}  \label{Remark1}
We can indeed obtain the $C^0$ estimate
\[ \epsilon \leq \underline{u} \leq u \leq C_0 \quad \text{on  } \overline{\Omega_{\epsilon}} \]
for any admissible solution $u \geq \underline{u}$ to \eqref{eqn10}, where $\epsilon > 0$ is any sufficiently small constant, $\psi$ is any prescribed positive function on $\overline{\Omega_{\epsilon}}$, and
$C_0$ is a positive constant independent of $\epsilon$ and $\psi$. In fact, we can pick a ball $B^{\sigma}_R (a)$ containing all $\Gamma_{\epsilon} \times \{ \epsilon \}$ for sufficiently small $\epsilon$ and pick $\sigma = 0$ in the proof of Lemma \ref{lemma5-1}.
\end{rem}

\subsection{Boundary gradient estimate}

\vspace{2mm}

For $\sigma \in (0,  1)$,
let $B^{\sigma} = B^{\sigma}_R = B^{\sigma}_R(b)$ be a ball in $\mathbb{R}^{n+1}$ of radius $R$ centered at $b = (b', \sigma R)$  and $S^{\sigma} = S^{\sigma}_R = \partial B^{\sigma}_R \cap \mathbb{R}^{n+1}_+$. Then $\kappa_i[S^{\sigma}] = \sigma$ for all $i$ with respect to its inward normal by \eqref{eqn14}.

\begin{lemma} \label{LemmaB}
For $\epsilon > 0$, let $\sigma \in (0, 1)$ be a constant satisfying \eqref{eqn19}. Let $B_R^{\sigma} (b)$ be a ball such that $b' \notin \overline{\Omega_{\epsilon}}$ and $\mbox{dist}(b', \Gamma_{\epsilon}) > \frac{\epsilon}{\sigma}$. If $B_R^{\sigma}(b) \cap (\Omega_{\epsilon} \times \{ \epsilon \}) = \emptyset$, then any admissible hypersurface $\Sigma^{\epsilon} = \{ (x, u(x) ) | x \in \Omega_{\epsilon} \}$ with $u \geq \underline{u}$ to \eqref{eqn10} satisfies $B^{\sigma}_R(b) \cap \Sigma^{\epsilon} = \emptyset$.
\end{lemma}

\begin{proof}
Suppose that $B^{\sigma} \cap (\Omega_{\epsilon} \times \{ \epsilon \}) = \emptyset$ and $B^{\sigma} \cap \Sigma^{\epsilon} \neq \emptyset$. Shrink $B^{\sigma}$ by homothetic dilations from $(b', 0)$ until $B^{\sigma} \cap \Sigma^{\epsilon} = \emptyset$. Then reverse the procedure until $S^{\sigma}$ first touches $\Sigma^{\epsilon}$ at some point $(x_0, u(x_0))$ where $x_0 \in \Omega_{\epsilon}$. Since $\Sigma^{\epsilon}$ is a $C^2$ graph,  $(x_0, u(x_0))$ must lie on the lower half of $S^{\sigma}$ (not including the equator). Note that $S^{\sigma}$ is locally a graph around $x_0$. Thus we reach a contradiction by Lemma \ref{LemmaA}.
\end{proof}

We have the following lemma for boundary gradient estimate.
\begin{lemma} \label{Lemma1-2}
Let $\epsilon$ be a sufficiently small constant which satisfies the compatibility conditions \eqref{eq1-4}. Then any admissible solution $u \geq \underline{u}$ to \eqref{eqn10} satisfies
\[ \frac{1}{\nu^{n+1}} <  \, \Big(\sigma - \frac{\sqrt{1 - \sigma^2}}{r_0^{\epsilon}} \,\epsilon - \frac{1 + \sigma}{(r_0^{\epsilon})^2}\,\epsilon^2\Big)^{-1} \quad\quad\mbox{on} \quad \Gamma_{\epsilon}. \]
\end{lemma}

\begin{proof}
The proof can be found in \cite{GS10} which applies Lemma \ref{LemmaB}.
\end{proof}

\vspace{2mm}

\subsection{Global gradient estimate}

We first write \eqref{eq1-1} as
\begin{equation} \label{eq3.0}
\sigma_k^{\frac{1}{k}} ( \kappa ) = f (\kappa) = F (A [u]) =  G (D^2 u, D u, u) = \psi (x, u).
\end{equation}
For convenience, we denote
\[ f_i = \frac{\partial f}{\partial \kappa_i}, \quad F^{ij} = \frac{\partial F}{\partial a_{ij}}, \quad
  G^{st} = \frac{\partial G}{\partial u_{st}},  \quad G^s = \frac{\partial G}{\partial u_s}, \quad G_u = \frac{\partial G}{\partial u}. \]
Differentiate \eqref{eq3.0}, we obtain
\begin{equation} \label{eq2.4}
G^{st} u_{st1} = \psi_{x_1} + \psi_u u_1 - G^s u_{s1} - G_u u_1.
\end{equation}

\begin{lemma} \label{Lemma1}
We have
\begin{equation*}
 G^{st} = \frac{u}{ w} F^{ij} \gamma^{is} \gamma^{t j},
\end{equation*}
\begin{equation*}
G^s  =  - \frac{u_s}{ w^2}  F^{ij} a_{ij}
- \frac{2 (w \gamma^{is} u_q + u_i \gamma^{q s}) }{ w (1 + w)} F^{i j} a_{q j} + \frac{2}{w^2} F^{ij} \gamma^{i s} u_j,
\end{equation*}
and
\begin{equation*}
G_u  = \frac{1}{u} \big( F^{ij} a_{ij} - \frac{1}{w} \sum f_i \big).
\end{equation*}
\end{lemma}
\begin{proof}
Since
\[  G( D^2 u, D u, u ) =  F \Big( \frac{1}{w} \big( u \gamma^{ik} u_{k l} \gamma^{l j} + \delta_{ij} \big) \Big),   \]
by direct computation,
\begin{equation*}
 G^{st}  = \,\frac{\partial F}{\partial a_{ij}} \frac{\partial a_{ij}}{\partial u_{st}} = \frac{u}{ w} F^{ij} \gamma^{is} \gamma^{t j},
\end{equation*}
\begin{equation*}
G_u = \frac{\partial F}{\partial a_{ij}} \frac{\partial a_{ij}}{\partial u} = F^{ij} \frac{1}{w} \gamma^{ik} u_{k l} \gamma^{lj} = \frac{1}{u} \Big( F^{ij} a_{ij} - \frac{1}{w} \sum f_i \Big),
\end{equation*}
and
\begin{equation*}
   G^s  = \, \frac{\partial F}{\partial a_{ij}} \frac{\partial a_{ij}}{\partial u_s}
         =  F^{ij} \Big( - \frac{ u_s}{w^3} \big( u \gamma^{ik} u_{kl} \gamma^{lj} + \delta_{ij} \big) + \frac{2 u }{w} \frac{ \partial \gamma^{ik}}{\partial u_s} u_{kl} \gamma^{lj} \Big).
\end{equation*}
Note that
\[ \frac{ \partial \gamma^{ik}}{\partial u_s} = - \gamma^{ip}\, \frac{\partial \gamma_{pq}}{\partial u_s} \, \gamma^{qk}, \]
\begin{equation*}
\frac{\partial \gamma_{p q}}{\partial u_s} =
  \frac{\delta_{ps} u_q + \delta_{q s} u_p }{1 + w} - \frac{u_p u_q u_s}{(1 + w)^2 w}
=  \frac{\delta_{p s} u_q + u_p \gamma^{q s}}{1 + w },
\end{equation*}
and
\[ \gamma^{ip} \,u_p = \frac{u_i}{w}, \]
we thus have
\begin{equation*}
G^s =  - \frac{u_s}{ w^2} \, F^{ij} a_{ij}
- \frac{2 (w \,\gamma^{is}\, u_q + u_i \,\gamma^{q s}) }{ w (1 + w)} F^{i j} a_{q j} + \frac{2}{w^2} F^{ij} \gamma^{i s} u_j.
\end{equation*}
\end{proof}

Consider the test function
\begin{equation*}
\Phi =   \ln  |D u| +  A u,
\end{equation*}
where $A$ is a positive constant to be determined. 
Assume the maximum of $\Phi$ is attained at $x^0 = (x_1, \ldots, x_n) \in \Omega_{\epsilon}$. Choose the Euclidean coordinate frame ${\partial}_1, \ldots, {\partial}_n$ around $x^0$ such that at $x^0$,
\[ u_1 = |D u| \quad \mbox{and} \quad u_{\alpha} = 0 \quad \mbox{for} \quad \alpha = 2, \ldots, n. \]
We may assume that $|D u| > 1$, since otherwise we are done.
By simple calculation, we immediately obtain
\begin{equation} \label{eq3.4}
 \gamma^{ik} = \delta_{ik} - \frac{u_i u_k}{w ( 1 + w )} = \left\{ \begin{aligned}
 & 1/w, \quad \mbox{if} \quad  i = k = 1, \\
 & \delta_{ik},  \quad \mbox{otherwise}.
 \end{aligned} \right.
\end{equation}

Then $\ln  u_1 + A u$ achieves its maximum at $x^0$, at which, we have
\begin{equation} \label{eq2.1}
\frac{u_{1i}}{u_1} +  A u_i = 0,
\end{equation}

\begin{equation} \label{eq2.2}
\frac{G^{ij} u_{1ij}}{u_1} - \frac{ G^{ij} u_{1i} u_{1j} }{u_1^2} + A G^{ij} u_{ij} \leq 0.
\end{equation}

From \eqref{eq2.1}, we have
\begin{equation} \label{eq2.9}
u_{11} = - A u_1^2  \quad \mbox{and}  \quad u_{1 \alpha} = 0 \quad \mbox{for} \quad \alpha = 2, \ldots, n.
\end{equation}
We may rotate ${\partial}_2, \ldots, {\partial}_n$ such that at $x^0$, $\big\{ u_{ij} \big\}$ is diagonal, and so is $\{ a_{ij} \}$:
\begin{equation} \label{eq2.3}
 a_{ij} = \frac{1}{w} \big( \delta_{ij} + u \gamma^{ik} u_{kl} \gamma^{lj} \big) = \left\{ \begin{aligned}
 & \frac{1}{w} \Big( 1 + \frac{u u_{11}}{w^2} \Big), \quad \mbox{if} \quad i = j = 1, \\
 &  \frac{1}{w} \big( 1 + u u_{ii} \big) \delta_{ij}, \quad \mbox{otherwise}.
  \end{aligned} \right.
\end{equation}
Consequently, $\{F^{ij}\}$ is also diagonal at $x^0$.

By Lemma \ref{Lemma1} and \eqref{eq3.4}, we can see that $\{G^{ij}\}$ is diagonal at $x^0$,
\begin{equation} \label{eq2.7}
 G^{ij} =  \left\{ \begin{aligned}
 & \frac{u}{w^3} F^{11}, \quad \mbox{if} \quad i = j = 1, \\
 &  \frac{u}{w} F^{ii} \delta_{ij}, \quad \mbox{otherwise}.
  \end{aligned} \right.
\end{equation}

By Lemma \ref{Lemma1}, \eqref{eq3.4} and \eqref{eq2.3}, we have
\begin{equation}  \label{eq2.5}
 - G^s u_{s1} - G_u u_1
=  \frac{2 u u_1 u_{11}^2}{w^5} F^{11} + \frac{u_1}{u w} \sum F^{ii} + \Big( \frac{u_1 u_{11}}{w^2} - \frac{u_1}{u} \Big) \psi,
\end{equation}
\begin{equation} \label{eq2.8}
G^{ij} u_{ij} =  \frac{u}{w} F^{ii} \gamma^{ii} \gamma^{ii}  u_{ii} =  \psi - \frac{1}{w} \sum F^{ii}.
\end{equation}

By \eqref{eq2.2}, \eqref{eq2.4}, \eqref{eq2.7}, \eqref{eq2.5} and \eqref{eq2.8}, we have
\begin{equation} \label{eq2.6}
\begin{aligned}
& \Big( \frac{2 u }{w^5} - \frac{u}{u_1^2 w^3} \Big) F^{11} u_{11}^2 - \frac{1}{w} \Big( A - \frac{1}{u} \Big) \sum F^{ii} \\
& + \frac{\psi_{x_1}}{u_1} + \psi_u + \Big( A + \frac{u_{11}}{w^2} - \frac{1}{u} \Big) \psi \leq 0.
\end{aligned}
\end{equation}

By \eqref{eq2.3} and \eqref{eq2.9},
\[ a_{11} =  \frac{1}{w} \Big( 1 + \frac{u u_{11}}{w^2} \Big) = \frac{1}{w} \Big( 1 - \frac{A u  u_1^2}{w^2} \Big) < 0  \]
if $A$ is chosen sufficiently large (which depends on $\epsilon$). It follows that
\begin{equation*}
\begin{aligned}
F^{11} =   \,& \frac{1}{k} \sigma_k^{\frac{1}{k} - 1}  \sigma_{k-1}(a_{22}, \ldots, a_{nn})  \\
=  \,& \frac{1}{k} \sigma_k^{\frac{1}{k} - 1} \Big(\sigma_{k-1}
 - a_{11} \sigma_{k-2}(a_{22}, \ldots, a_{nn}) \Big)  \\
\geq \,&   \frac{1}{k} \sigma_k^{\frac{1}{k} - 1} \sigma_{k-1}.
\end{aligned}
\end{equation*}
Then by Newton-Maclaurin inequality, we have
\begin{equation} \label{eq2.10}
 c(n, k) \leq \sum F^{ii} = \frac{n - k + 1}{k} \sigma_k^{\frac{1}{k} - 1} \sigma_{k - 1} \leq (n - k + 1) F^{11},
\end{equation}
where $c(n, k)$ is a positive constant.

Choosing $A$ sufficiently large, by \eqref{eq2.6}, \eqref{eq2.10}, \eqref{eq2.9} and assumption \eqref{eq1-2}, we obtain an upper bound for $u_1$.

\vspace{4mm}

\section{Global curvature estimate}

\vspace{4mm}

In this section, we will derive second order estimate if we know them on the boundary. For a hypersurface $\Sigma$, let $g$ and $\nabla$ denote the induced metric and Levi-Civita connection on $\Sigma$ induced from $\mathbb{H}^{n+1}$, while $\tilde{g}$ and $\tilde{\nabla}$ be the ones induced from $\mathbb{R}^{n+1}$. The Christoffel symbols with respect to $\nabla$ and $\tilde{\nabla}$ are related by the formula
\[ \Gamma_{ij}^k = \tilde{\Gamma}_{ij}^k - \frac{1}{u} (u_i \delta_{kj} + u_j \delta_{ik} - \tilde{g}^{kl} u_l \tilde{g}_{ij}). \]
Consequently, for any $v \in C^2(\Sigma)$ and in any local frame on $\Sigma$,
\begin{equation} \label{eqC-3}
\nabla_{ij} v = (v_i)_j - \Gamma_{ij}^k v_k = \tilde{\nabla}_{ij} v + \frac{1}{u}( u_i v_j + u_j v_i - \tilde{g}^{kl} u_l v_k \tilde{g}_{ij} ).
\end{equation}

\begin{lemma}  \label{Lemma0-3}
In $\mathbb{R}^{n+1}$,
\begin{equation} \label{eq0-1}
\tilde{g}^{kl} u_k u_l  = |\tilde\nabla u|^2 = 1 - (\nu^{n+1})^2,
\end{equation}
\begin{equation}  \label{eq0-2}
\tilde{\nabla}_{ij} u = \tilde{h}_{ij} \nu^{n+1} \quad \mbox{and} \quad \tilde{\nabla}_{ij} x_{k} = \tilde{h}_{ij} \nu^{k}, \quad k = 1, \ldots, n,
\end{equation}
\begin{equation}  \label{eq0-4}
(\nu^{n+1})_i = - \tilde{h}_{ij} \,\tilde{g}^{j k} u_k,
\end{equation}
\begin{equation} \label{eq0-5}
\tilde{\nabla}_{ij} \nu^{n+1} = - \tilde{g}^{kl} ( \nu^{n+1} \tilde{h}_{il} \tilde{h}_{kj} + u_l \tilde{\nabla}_k \tilde{h}_{ij} ),
\end{equation}
where $\tau_1, \ldots, \tau_n$ is any local frame on $\Sigma$.
\end{lemma}
\begin{proof}
The identities in this Lemma can be found in \cite{GS11} and the proof can be found in \cite{Sui2019}.
\end{proof}

\begin{lemma}  \label{LemmaC-1}
Let $\Sigma$ be an admissible hypersurface in $\mathbb{H}^{n+1}$ satisfying equation \eqref{eq1-1}. Then in a local orthonormal frame on $\Sigma$,
\begin{equation}  \label{eqC-5}
\begin{aligned}
F^{ij} \nabla_{ij} \nu^{n+1}
= & - \nu^{n+1} F^{ij} h_{ik} h_{kj} + \big( 1 + (\nu^{n+1})^2 \big) F^{ij} h_{ij} - \nu^{n+1} \sum f_i \\ & - \frac{2}{u^2} F^{ij} h_{jk} u_i u_k + \frac{2 \nu^{n+1}}{u^2} F^{ij} u_i u_j - \frac{u_k}{u} \psi_k.
\end{aligned}
\end{equation}
\end{lemma}
\begin{proof}
The proof can be found in \cite{Sui2019}, which utilizes the above identities.
\end{proof}

Now we state the main theorem in this section on global curvature estimate, which is equivalent to global second order estimate.
\begin{thm}
Let $\Sigma = \{( x, u (x) ) \,\vert\, x \in \Omega_{\epsilon} \}$ be an admissible $C^4$ graph in $\mathbb{H}^{n + 1}$ satisfying \eqref{eq1-1} for some positive function $\psi(x, u) \in C^2 (\mathbb{H}^{n + 1})$.
Then there exists a positive constant $C$ depending only on $n$, $k$, $\epsilon$, $\Vert u \Vert_{C^1(\Omega_{\epsilon})}$ and $\Vert\psi\Vert_{C^2}$ such that
\[ \sup\limits_{\substack{ x \in \Omega_{\epsilon} \\  i = 1, \ldots, n}}  \kappa_i (x)  \leq C \Big( 1 + \sup\limits_{\substack{x \in \Gamma_{\epsilon} \\ i = 1, \ldots, n}}  \kappa_i (x) \Big).  \]
\end{thm}
\begin{proof}
First, note that
\[ \nu^{n+1} = \frac{1}{\sqrt{1 + |D u|^2}} \geq 2 a > 0 \quad \mbox{on} \,\, \Sigma \]
for some positive constant $a$.
Let $\kappa_{\max} ({ \bf x })$ be the largest principal curvature of $\Sigma$ at ${\bf x}$. Consider
\begin{equation*}
M_0 = \sup\limits_{{\bf x} \in\Sigma}
\,\frac{\kappa_{\max\,}({\bf x})}{{\nu}^{n+1} - a} e^{\frac{\beta}{u}},
\end{equation*}
where $\beta$ is a positive constant to be determined. Assume $M_0 > 0$ is attained at an interior point ${ \bf x}_0 \in \Sigma$.
Let $\tau_1, \ldots, \tau_n$ be a local orthonormal frame about
${ \bf x}_0$ such that $h_{ij}({\bf x}_0) = \kappa_i \,\delta_{ij}$, where
$\kappa_1 \geq \ldots \geq \kappa_n$ are the hyperbolic principal curvatures of
$\Sigma$ at ${\bf x}_0$.
Thus, $\ln h_{11} - \ln ( {\nu}^{n+1} - a ) + \frac{\beta}{u}$ has a
local maximum at ${\bf x}_0$, at which,
\begin{equation} \label{eq2G-1}
\frac{h_{11i}}{h_{11}} - \frac{\nabla_i \nu^{n + 1}}{\nu^{ n + 1 } - a} - \beta \frac{u_i}{u^2} = 0,
\end{equation}
\begin{equation} \label{eq2G-2}
\frac{F^{ii} h_{11ii}}{h_{11}} - \frac{F^{ii} h_{11i}^2}{h_{11}^2} - \frac{F^{ii} \nabla_{ii} \nu^{n + 1}}{\nu^{n + 1} - a} + \frac{F^{ii} (\nu^{n + 1})_i^2}{(\nu^{n + 1} - a)^2} - \beta F^{ii} \frac{\nabla_{ii} u}{u^2} + \beta F^{ii} \frac{2 u_i^2}{u^3} \leq 0.
\end{equation}

Differentiate equation \eqref{eq1-1} twice,
\begin{equation}  \label{eq2G-3}
F^{ii} h_{ii11} + F^{ij, rs} h_{ij1} h_{rs1} = \psi_{11}  \geq  - C \kappa_1.
\end{equation}

By Gauss equation, we have the following commutation formula,
\begin{equation} \label{eq2G-4}
h_{ii11} = h_{11ii} + ( \kappa_i \kappa_1 - 1 )( \kappa_i - \kappa_1 ).
\end{equation}

By \eqref{eq0-1}, we have
\begin{equation} \label{eq2G-13}
\tilde{g}^{kl} u_k u_l  = \frac{\delta_{kl}}{u^2} u_k u_l = 1 - (\nu^{n+1})^2.
\end{equation}
By \eqref{eqC-3}, \eqref{eq0-2}, \eqref{eq2G-13} and \eqref{eq0-3}, we have
\begin{equation} \label{eq0-9}
- \beta F^{ii} \frac{\nabla_{ii} u}{u^2} + \beta F^{ii} \frac{2 u_i^2}{u^3} = \frac{\beta}{u} \sum F^{ii} - \beta \psi \frac{\nu^{n+1}}{u}.
\end{equation}

Combining \eqref{eq2G-2}, \eqref{eq2G-4}, \eqref{eq2G-3}, \eqref{eqC-5} and \eqref{eq0-9} yields,
\begin{equation} \label{eq2G-5}
\begin{aligned}
& \Big( \kappa_1 - \frac{\beta \nu^{n + 1}}{u} \Big) \psi  - C  + \Big( \frac{\beta}{u} + \frac{a}{\nu^{n+1} - a} \Big) \sum f_i \\
& + \frac{a}{\nu^{n+1} - a} \sum f_i \kappa_i^2   + \frac{2}{\nu^{n+1} - a}  \sum f_i \kappa_i \frac{u_i^2}{u^2} - \frac{2 \nu^{n+1}}{\nu^{n+1} - a} \sum f_i \frac{u_i^2}{u^2} \\
& - \frac{F^{ij, rs} h_{ij1} h_{rs1}}{\kappa_1} - \frac{F^{ii} h_{11i}^2}{\kappa_1^2} + \frac{F^{ii} (\nu^{n+1})_i^2}{(\nu^{n+1} - a)^2} \leq 0.
\end{aligned}
\end{equation}

Let $\theta \in (0, 1)$ be a constant which will be determined later. Using the idea of Jin-Li \cite{JL05}, we divide our discussion into two cases.

\vspace{2mm}

Case (i). Assume $\kappa_n \leq - \theta \kappa_1$.
By \eqref{eq2G-1} and Cauchy-Schwartz inequality,
\[
\begin{aligned}
- \frac{F^{ii} h_{11i}^2}{\kappa_1^2} + \frac{F^{ii} (\nu^{n+1})_i^2}{(\nu^{n+1} - a)^2}
\geq   - \delta_1 \frac{F^{ii} (\nu^{n+1})_i^2}{(\nu^{n+1} - a)^2} - \Big( 1 + \frac{1}{\delta_1} \Big) \beta^2 f_i \frac{u_i^2}{u^4},
\end{aligned}
\]
where $\delta_1$ is a positive constant to be determined later.
By \eqref{eq0-4} and \eqref{eq0-3},
\begin{equation} \label{eq2G-11}
(\nu^{n+1})_i = \frac{u_i}{u} (\nu^{n + 1} - \kappa_i).
\end{equation}
In view of \eqref{eq2G-13}, we have
\begin{equation} \label{eq2G-8}
\begin{aligned}
& - \frac{F^{ii} h_{11i}^2}{\kappa_1^2} + \frac{F^{ii} (\nu^{n+1})_i^2}{(\nu^{n+1} - a)^2} \\
\geq  &  - \frac{2 \delta_1}{(\nu^{n+1} - a)^2} \sum f_i \kappa_i^2 - \Big( \frac{2 \delta_1}{(\nu^{n + 1} - a)^2} + \frac{\beta^2}{u^2} \big(1 + \frac{1}{\delta_1} \big) \Big) \sum f_i.
\end{aligned}
\end{equation}

By \eqref{eq2G-13} and Cauchy-Schwartz inequality,
\begin{equation}  \label{eq2G-9}
\begin{aligned}
 & \frac{2}{\nu^{n+1} - a}  \sum f_i \kappa_i \frac{u_i^2}{u^2} -  \frac{2 \nu^{n+1}}{\nu^{n+1} - a}  \sum f_i \frac{u_i^2}{u^2} \\
\geq & - \frac{2}{\nu^{n+1} - a}  \sum f_i |\kappa_i| - \frac{2}{\nu^{n+1} - a}  \sum f_i \\
\geq & -  \frac{1}{ \delta_2 (\nu^{n+1} - a)}  \sum f_i  - \frac{\delta_2}{(\nu^{n+1} - a)} \sum f_i \kappa_i^2 - \frac{2}{\nu^{n+1} - a} \sum f_i,
\end{aligned}
\end{equation}
where $\delta_2$ is a positive constant to be determined later.

By assumption,
\begin{equation} \label{eq2G-7}
\sum f_i \kappa_i^2 \geq f_n \kappa_n^2 \geq \frac{1}{n} \sum f_i \theta^2 \kappa_1^2 = \frac{\theta^2}{n} \kappa_1^2 \sum f_i.
\end{equation}
Therefore, by \eqref{eq2G-8} with $\delta_1 = \frac{a^2}{8}$, \eqref{eq2G-9} with $\delta_2 = \frac{a}{4}$ and \eqref{eq2G-7}, inequality \eqref{eq2G-5} reduces to
\[
\begin{aligned}
&  \Big( \frac{\beta}{u} + \frac{a}{\nu^{n+1} - a} - \frac{2 \delta_1}{(\nu^{n + 1} - a)^2} - \frac{\beta^2}{u^2} \big(1 + \frac{1}{\delta_1} \big)
- \frac{1}{ \delta_2 (\nu^{n+1} - a)} - \frac{2}{\nu^{n+1} - a}  \Big) \sum f_i \\
& + \Big( \kappa_1 - \frac{\beta \nu^{n + 1}}{u} \Big) \psi - C  + \frac{a}{2(\nu^{n+1} - a)} \frac{\theta^2}{n} \kappa_1^2 \sum f_i \leq 0.
\end{aligned}
\]
Also note that $\sum f_i \geq c(n, k)$ by Newton-Maclaurin inequality, we thus obtain an upper bound for $\kappa_1$.

\vspace{2mm}

Case (ii).  Assume $\kappa_n > - \theta \kappa_1$.
Denote
\[ J = \{  i \,| \, f_1 \geq \theta^2 f_i  \},\quad \quad L = \{  i \,| \, f_1 < \theta^2 f_i  \}. \]

By \eqref{eq2G-1}, Cauchy-Schwartz inequality, \eqref{eq2G-11} and \eqref{eq2G-13},
\begin{equation} \label{eq2G-10}
\begin{aligned}
& - \sum_{i \in J} \frac{F^{ii} h_{11i}^2}{\kappa_1^2} + \frac{F^{ii} (\nu^{n+1})_i^2}{(\nu^{n+1} - a)^2}  \\
 \geq  & - \delta_3 \frac{F^{ii} (\nu^{n+1})_i^2}{(\nu^{n+1} - a)^2} - \Big( 1 + \frac{1}{\delta_3} \Big) \beta^2 \sum_{i \in J} f_i \frac{u_i^2}{u^4} \\
\geq &   - \frac{2 \delta_3}{(\nu^{n+1} - a)^2} \sum f_i - \frac{2 \delta_3}{(\nu^{n+1} - a)^2} \sum f_i \kappa_i^2 - \Big( 1 + \frac{1}{\delta_3} \Big)  \frac{\beta^2 f_1}{\theta^2 u^2}.
\end{aligned}
\end{equation}

Using an inequality of Andrews \cite{And} and Gerhardt \cite{Ger},
\[ - F^{ij, rs} h_{ij1} h_{rs1}  \geq  \sum\limits_{i \neq j} \frac{f_i - f_j}{\kappa_j - \kappa_i} h_{ij1}^2 \geq  2 \sum\limits_{i \geq 2} \frac{f_i - f_1}{\kappa_1 - \kappa_i} h_{i11}^2  \]
and taking $\theta = \frac{1}{2}$, we have
\begin{equation} \label{eq2G-12}
\begin{aligned}
  - \frac{F^{ij, rs} h_{ij1} h_{rs1}}{\kappa_1} - \sum_{i \in L} \frac{F^{ii} h_{11i}^2}{\kappa_1^2}
\geq  \frac{2 (1 - \theta)}{\kappa_1^2} \sum_{i \in L} f_i h_{11i}^2 - \sum_{i \in L} \frac{F^{ii} h_{11i}^2}{\kappa_1^2} = 0.
\end{aligned}
\end{equation}

By \eqref{eq2G-12}, \eqref{eq2G-10} with $\delta_3 = \frac{a^2}{8}$ and \eqref{eq2G-9} with $\delta_2 = \frac{a}{4}$, \eqref{eq2G-5} reduces to
\[
\begin{aligned}
& \Big( \frac{\beta}{u} + \frac{a}{\nu^{n+1} - a}  - \frac{2 \delta_3}{(\nu^{n+1} - a)^2} - \frac{1}{ \delta_2 (\nu^{n+1} - a)} - \frac{2}{\nu^{n+1} - a} \Big) \sum f_i \\
& + \Big( \kappa_1 - \frac{\beta \nu^{n + 1}}{u} \Big) \psi - C  + \frac{a}{2(\nu^{n+1} - a)} \sum f_i \kappa_i^2  - \Big( 1 + \frac{1}{\delta_3} \Big)  \frac{\beta^2 f_1}{\theta^2 u^2} \leq 0.
\end{aligned}
\]
Taking $\beta$ sufficiently large, we obtain an upper bound for $\kappa_1$.
\end{proof}

\vspace{4mm}

\section{Second order boundary estimate}

\vspace{4mm}

\subsection{Tangential-normal second derivative estimate}
For an arbitrary point on $\Gamma_{\epsilon}$, we may assume it to be the origin of $\mathbb{R}^n$. Choose a coordinate system so that the positive $x_n$ axis points to the interior normal of $\Gamma_{\epsilon}$ at $0$. There exists a uniform constant $r > 0$ such that $\Gamma_{\epsilon} \cap B_r (0)$ can be represented as a graph
\[ x_n = \rho ( x' ) = \frac{1}{2} \sum\limits_{s, t < n} B_{s t} x_{s} x_{t} + O ( |x'|^3 ), \quad x' = (x_1, \ldots, x_{n - 1}).  \]

Let $u \in C^3 (\overline{\Omega_{\epsilon}})$ be an admissible solution to \eqref{eq1-1} satisfying $u \geq \underline{u}$ in $\Omega_{\epsilon}$ and $u = \epsilon$ on $\Gamma_{\epsilon}$. For the tangential-normal second derivative estimate, consider for $t < n$,
\[ W = u_t +  u_n \rho_{t} - \frac{1}{2} \sum_{s < n} u_s^2. \]
By direct calculation,
\begin{equation} \label{eq3.21}
D_i W = u_{t i} + u_{n i} \rho_t + u_n \rho_{t i} - \sum_{s < n} u_s u_{s i},
\end{equation}
\begin{equation} \label{eq3.22}
\begin{aligned}
D_{ij} W =  u_{t i j} + u_{nij} \rho_t + u_{n i} \rho_{t j} + u_{n j} \rho_{t i} + u_n \rho_{tij}
 - \sum_{s < n} u_s u_{sij} - \sum_{s < n} u_{si} u_{sj}.
\end{aligned}
\end{equation}

Following \cite{Ivochkina89,Lin-Trudinger94,Ivoch-Lin-Tru96}, we write equation \eqref{eq1-1} in the following equivalent form
\begin{equation} \label{eq2B-1}
\mathcal{G}( D^2 u, D u, u ) =  F \Big( u \gamma^{l i} u_{ij} \gamma^{j m} + \delta_{l m} \Big) = \psi (x, u) w = \Psi(x, u, D u).
\end{equation}
Denote
\[ \mathcal{G}^{ij} = \frac{\partial\mathcal{G}}{\partial u_{ij}}, \quad \mathcal{G}^i = \frac{\partial\mathcal{G}}{\partial u_i}, \quad \mathcal{G}_u = \frac{\partial\mathcal{G}}{\partial u}, \quad \Psi^i = \frac{\partial \Psi}{\partial u_i},   \]
and
\[ L  =  \mathcal{G}^{ij} D_{ij} - \Psi^i D_i. \]

In order to give an estimation for $LW$, we need to choose a special local frame, which was utilized by Ivochkina \cite{Ivochkina89}. For fixed $x_0 \in \Omega_{\epsilon}$, choose a local frame $\tau_1, \ldots, \tau_n$ around $x_0$ on $\Omega_{\epsilon}$ such that
\[ \tau_{\alpha} + u_{\tau_\alpha} \partial_{n + 1}, \quad \alpha = 1, \ldots, n \]
is a local orthonormal frame around $(x_0, u(x_0))$ on $\Sigma^{\epsilon} = \{ (x, u(x)) \,|\, x \in \Omega_{\epsilon} \}$ and in addition they are principal directions at $(x_0, u(x_0))$ on $\Sigma^{\epsilon}$. In fact, we can choose
\[ \tau_{\alpha} = P_{\alpha l} u  \gamma^{li} \partial_i, \quad\quad \alpha = 1, \ldots, n,  \]
where $P = (P_{ij})$ is a constant orthogonal matrix such that
\[ P_{\alpha l} \, \frac{ u \gamma^{l i} u_{ij} \gamma^{j m} + \delta_{lm} }{w} (x_0) \, P_{\beta m} \]
is diagonal. Then we can verify that
\[\begin{aligned}
& \big\langle \tau_{\alpha} + u_{\tau_{\alpha}} \partial_{n + 1},  \tau_{\beta} + u_{\tau_{\beta}} \partial_{n + 1} \big\rangle =  \frac{1}{u^2} \big( \tau_{\alpha} \cdot \tau_{\beta} + u_{\tau_{\alpha}} u_{\tau_{\beta}} \big) \\
= & P_{\alpha l} \gamma^{l i} \big( \delta_{ij} + u_i u_j \big) \gamma^{j m} P_{\beta m} = \delta_{\alpha \beta},
\end{aligned}\]
where $\langle\, , \, \rangle$ is the inner product in hyperbolic space, and $\cdot$ is the inner product in Euclidean space. In addition,
\[ u_{\tau_\alpha} = u P_{\alpha l} \gamma^{l i} u_i, \quad \quad  u_{\tau_\alpha \tau_\beta} = u^2 P_{\alpha l} P_{\beta m} \gamma^{l i} \gamma^{m j} u_{ij}, \]
\begin{equation} \label{eq2B-7}
\begin{aligned}
& a_{\tau_\alpha \tau_\beta} =  h_{\tau_\alpha \tau_\beta} = P_{\alpha l} u \gamma^{li} P_{\beta m} u \gamma^{m j} h_{ij}  \\
= & P_{\alpha l} \Big( \frac{\delta_{lm}}{w} + \frac{ u \gamma^{l i} \gamma^{m j} u_{ij}}{w} \Big) P_{\beta m} = \frac{\delta_{\alpha \beta}}{w} + \frac{u_{\tau_\alpha \tau_\beta}}{u w},
\end{aligned}
\end{equation}
and $a_{\tau_\alpha \tau_\beta}(x_0)$ is diagonal.

Throughout this subsection, Greek letter $\alpha, \beta, \ldots$ are from $1$ to $n$.
Denote
\[ \mathcal{A}_{\alpha \beta} = P_{\alpha l} \big( \delta_{lm} +  u \gamma^{l i} \gamma^{m j} u_{ij} \big) P_{\beta m} = \delta_{\alpha \beta} + \frac{u_{\tau_\alpha \tau_\beta}}{u}.  \]
Equation \eqref{eq2B-1} can also be expressed as
\begin{equation} \label{eq2B-1-1}
\mathcal{G}( D^2 u, D u, u ) =  F ( \mathcal{A}_{\alpha \beta} ) = f(\lambda) = \Psi(x, u, D u).
\end{equation}
Then denote
\[ F^{\alpha\beta} = \frac{\partial F}{\partial \mathcal{A}_{\alpha\beta}}, \quad f_{\alpha} = \frac{\partial f}{\partial \lambda_{\alpha}}. \]
At $x_0$, we have
\[  \mathcal{A}_{\alpha \beta} = \lambda_{\alpha} \delta_{\alpha \beta}, \quad  F^{\alpha\beta} = f_{\alpha} \delta_{\alpha \beta}. \]
By direct calculation similar to Lemma \ref{Lemma1}, we have the following lemma.
\begin{lemma}  \label{Lemma2B}
At $x_0$, we have
\begin{equation*}
 \mathcal{G}^{ij}  =   u f_{\alpha} P_{\alpha l} \gamma^{l i} P_{\alpha m} \gamma^{m j},
\end{equation*}
\begin{equation*}
\mathcal{G}^i  =  - \frac{2 P_{\alpha l} \gamma^{l i} P_{\alpha q} u_{q} }{w} f_{\alpha} (\lambda_{\alpha} - 1),
\end{equation*}
\begin{equation*}
\mathcal{G}_u  = \frac{1}{u} \Big( \Psi - \sum f_{\alpha} \Big),
\end{equation*}
\[ \Psi^i = \psi(x, u) \frac{u_i}{w}.  \]
\end{lemma}
\begin{proof}
\begin{equation*}
 \mathcal{G}^{ij} = \frac{\partial F}{\partial\mathcal{A}_{\alpha\beta}} \frac{\partial \mathcal{A}_{\alpha\beta}}{\partial u_{ij}} = u F^{\alpha\beta} P_{\alpha l} \gamma^{li} \gamma^{m j} P_{\beta m} =  u f_{\alpha} P_{\alpha l} \gamma^{l i} P_{\alpha m} \gamma^{m j}.
\end{equation*}

\begin{equation*}
\mathcal{G}_u = F^{\alpha\beta} P_{\alpha l} \gamma^{l i} u_{ij} P_{\beta m} \gamma^{j m}  = \frac{1}{u} \Big( \Psi - \sum f_{\alpha} \Big).
\end{equation*}

\begin{equation*}
\mathcal{G}^s  =  \frac{\partial F}{\partial \mathcal{A}_{\alpha\beta}} \frac{\partial \mathcal{A}_{\alpha\beta}}{\partial u_s}
         = 2 u F^{\alpha\beta} P_{\alpha l} \frac{ \partial \gamma^{l i}}{\partial u_s} u_{ij} \gamma^{j m} P_{\beta m}.
\end{equation*}
Note that
\[ \frac{ \partial \gamma^{l i}}{\partial u_s} = - \gamma^{l p} \frac{\partial \gamma_{pq}}{\partial u_s}  \gamma^{q i}, \quad \quad
  \frac{\partial \gamma_{p q}}{\partial u_s}
=  \frac{\delta_{p s} u_q + u_p \gamma^{q s}}{1 + w },   \quad \quad
\gamma^{l p} u_p = \frac{u_l}{w}. \]
Therefore,
\begin{equation*}
\begin{aligned}
\mathcal{G}^s =  - 2 f_{\alpha} P_{\alpha l} \frac{\gamma^{ls} u_q w + u_l \gamma^{qs}}{(1 + w) w} P_{\alpha q} (\lambda_{\alpha} - 1)
 =  - \frac{2 P_{\alpha l} \gamma^{l s} P_{\alpha q} u_{q} }{w} f_{\alpha} (\lambda_{\alpha} - 1).
\end{aligned}
\end{equation*}
\end{proof}
Differentiating \eqref{eq2B-1-1}, we have
\begin{equation} \label{eq2B-2}
\mathcal{G}^{ij} u_{ijk} + \mathcal{G}^i u_{ik} + \mathcal{G}_u u_k = (\psi_{x_k} + \psi_u u_k) w + \Psi^i u_{ik}.
\end{equation}

By \eqref{eq3.21}, \eqref{eq3.22}, \eqref{eq2B-2} and Lemma \ref{Lemma2B}, we have
\begin{equation} \label{eq2B-8}
\begin{aligned}
L W = & (\psi_{x_t} + \psi_u u_t ) w + \rho_t (\psi_{x_n} + \psi_u u_n) w - \mathcal{G}_u (u_t + u_n \rho_t) \\
&  + 2 \mathcal{G}^{ij} u_{n i} \rho_{t j} + \mathcal{G}^{ij} u_n \rho_{tij} - \Psi^i u_n \rho_{t i} - \mathcal{G}^i D_i W + \mathcal{G}^i u_n \rho_{t i} \\
& - \sum_{s < n} (\psi_{x_s} w + \psi_u u_s w - \mathcal{G}_u u_s ) u_s  - \sum_{s < n} \mathcal{G}^{ij} u_{si} u_{sj} \\
\leq & C \sum f_{\alpha}  + 2 \mathcal{G}^{ij} u_{n i} \rho_{t j} - \mathcal{G}^i D_i W + \mathcal{G}^i u_n \rho_{t i} - \sum_{s < n} \mathcal{G}^{ij} u_{si} u_{sj}.
\end{aligned}
\end{equation}
By Lemma \ref{Lemma2B},
\begin{equation} \label{eq2B-4}
2 \mathcal{G}^{ij} u_{n i} \rho_{t j} =  2 f_{\alpha} P_{\alpha l} \gamma^{l i} (\lambda_{\alpha} - 1) P_{\alpha m} \gamma_{m n} \rho_{ti}
\leq   \delta_1 \sum f_{\alpha} \lambda_{\alpha}^2 + \frac{C}{\delta_1} \sum f_{\alpha},
\end{equation}
where $\delta_1$ is a positive constant to be determined later,
\begin{equation} \label{eq2B-5}
\begin{aligned}
& - \mathcal{G}^i D_i W + \mathcal{G}^i u_n \rho_{t i} \\
= & \frac{2 P_{\alpha q} u_q }{u w} f_{\alpha} (\lambda_{\alpha} - 1) D_{\tau_{\alpha}} W + \frac{2 P_{\alpha q} u_{q} }{w} f_{\alpha} ( - \lambda_{\alpha} + 1)  u_n P_{\alpha l} \gamma^{l i} \rho_{t i} \\
\leq & \delta_1 \sum f_{\alpha} \lambda_{\alpha}^2 + \frac{C}{\delta_1} \Big( \sum f_{\alpha} (D_{\tau_{\alpha}} W)^2 + \sum f_{\alpha} \Big),
\end{aligned}
\end{equation}
and
\begin{equation} \label{eq2B-6}
\begin{aligned}
& \sum_{s < n} \mathcal{G}^{ij} u_{is}  u_{js} = \frac{1}{u} f_{\alpha} (\lambda_{\alpha} - 1)^2 \sum_{s < n} ( P_{\alpha l} \gamma_{l s})^2 \\
 \geq & \frac{1}{2 u} \sum f_{\alpha} \lambda_{\alpha}^2 \sum_{s < n} ( P_{\alpha l} \gamma_{ls})^2 - C \sum f_{\alpha}.
\end{aligned}
\end{equation}
Taking \eqref{eq2B-4}--\eqref{eq2B-6} into \eqref{eq2B-8},
\begin{equation} \label{eq2B-9}
\begin{aligned}
 L W & \leq 2 \delta_1 \sum f_{\alpha} \lambda_{\alpha}^2 - \frac{1}{2 u} \sum f_{\alpha} \lambda_{\alpha}^2 \sum_{s < n} (P_{\alpha l} \gamma_{ls})^2 \\
& + \frac{C}{\delta_1} \sum f_{\alpha} + \frac{C}{\delta_1} \sum f_{\alpha} (D_{\tau_{\alpha}} W)^2.
\end{aligned}
\end{equation}

Using Ivochkina's method \cite{Ivochkina89}, we divide our discussion into two cases.

{\bf Case (i).}  Suppose for any $\alpha = 1, \ldots, n$,
\[ \sum_{s < n} (P_{\alpha l} \gamma_{ls})^2 \geq \epsilon_1^2,  \]
where $\epsilon_1$ is a positive constant to be determined. Picking $\delta_1 < \frac{\epsilon_1^2}{4 \sup_{\Omega_{\epsilon}} u}$, \eqref{eq2B-9} reduces to
\begin{equation} \label{eq2B-10}
 L W \leq  C \sum f_{\alpha} + C \sum f_{\alpha} (D_{\tau_{\alpha}} W)^2.
\end{equation}

{\bf Case (ii).} If for some $\beta \in \{ 1, \ldots, n \}$,
\[ \sum_{s < n} ( P_{\beta l} \gamma_{ls})^2 < \epsilon_1^2. \]
For any $\alpha \neq \beta$, consider the Laplace expansion along the $\alpha$th row
\[ \begin{aligned}
w = & \det (P_{\alpha l} \gamma_{l j}) \leq  \sum_{s < n} |P_{\alpha l} \gamma_{l s}| (n - 1)! w^{n - 1} + |P_{\alpha l} \gamma_{l n}| (n - 1)! \epsilon_1 w^{n - 2}.
\end{aligned} \]
Thus, we can pick any
\[ 0 < \epsilon_1 < \frac{1}{2 (n - 1)! (\sup{w})^{n - 2}}, \]
and obtain for $\alpha \neq \beta$,
\[  \sum_{s < n} (P_{\alpha l} \gamma_{ls})^2 > \epsilon_2^2 \quad \quad \mbox{with} \quad \epsilon_2 = \frac{1}{2 n! (\sup{w})^{n - 2}}. \]
Consequently, \eqref{eq2B-6} can be estimated as
\begin{equation}  \label{eq2B-3}
\begin{aligned}
 \sum_{s < n} \mathcal{G}^{ij} u_{is}  u_{js} \geq  \frac{\epsilon_2^2}{2 u} \sum_{\alpha \neq \beta} f_{\alpha} \lambda_{\alpha}^2 - C \sum f_{\alpha}.
\end{aligned}
\end{equation}

Next, we shall derive an inequality in place of \eqref{eq2B-4}.
Note that \eqref{eq2B-4} can be replaced by
\[
\begin{aligned}
 2 \mathcal{G}^{ij} u_{n i} \rho_{t j}
\leq \delta_2 \sum_{\alpha \neq \beta} f_{\alpha} \lambda_{\alpha}^2 + \frac{C}{\delta_2} \sum f_{\alpha}  + 2 f_{\beta} P_{\beta l} \gamma^{l i} \lambda_{\beta} P_{\beta m} \gamma_{m n} \rho_{ti},
\end{aligned}
\]
where $\delta_2$ is a positive constant to be determined, and
\[ f_{\beta} \lambda_{\beta} =  \Psi - \sum_{\alpha \neq \beta} f_{\alpha} \lambda_{\alpha}. \]
Therefore, \eqref{eq2B-4} can be replaced by
\begin{equation} \label{eq2B-11}
2 \mathcal{G}^{ij} u_{n i} \rho_{t j}
\leq  2 \delta_2 \sum_{\alpha \neq \beta} f_{\alpha} \lambda_{\alpha}^2 + \frac{C}{\delta_2} \sum f_{\alpha}.
\end{equation}
Similarly, we can replace \eqref{eq2B-5} by the following inequality.
\begin{equation} \label{eq2B-15}
\begin{aligned}
& - \mathcal{G}^i D_i W + \mathcal{G}^i u_n \rho_{t i} \\
\leq \, & \frac{2 P_{\beta q} u_{q} }{u w} f_{\beta} \lambda_{\beta}  D_{\tau_{\beta}} W
+ \frac{3 \delta_2 }{2}  \sum_{\alpha \neq \beta} f_{\alpha} \lambda_{\alpha}^2 + \frac{C}{\delta_2} \Big( \sum f_{\alpha} (D_{\tau_{\alpha}} W)^2 + \sum f_{\alpha} \Big).
\end{aligned}
\end{equation}

Now, we need to give an estimation for $\frac{2 P_{\beta q} u_{q} }{u w} f_{\beta} \lambda_{\beta}  D_{\tau_{\beta}} W$. We use Ivochkina's method \cite{Ivochkina89} to divide the discussion into two subcases.

\vspace{2mm}

{\bf Subcase (i).} Suppose $2 \sigma_{k - 1} (\lambda | \beta) > \sigma_{k - 1}$. Then
\[ \begin{aligned}
 \frac{2 P_{\beta q} u_{q} }{u w} f_{\beta} \lambda_{\beta}  D_{\tau_{\beta}} W
& =   \frac{2 P_{\beta q} u_{q} }{u w} \Psi D_{\tau_{\beta}} W - \frac{2 P_{\beta q} u_{q} }{u w} \sum_{\alpha \neq \beta} f_{\alpha} \lambda_{\alpha} D_{\tau_{\beta}} W \\
& \leq  C  |D W| + \frac{\delta_2}{2} \sum_{\alpha \neq \beta} f_{\alpha} \lambda_{\alpha}^2 + \frac{C}{\delta_2} \sum_{\alpha \neq \beta} f_{\alpha} (D_{\tau_{\beta}} W)^2 \\
& \leq C  |D W| + \frac{\delta_2}{2} \sum_{\alpha \neq \beta} f_{\alpha} \lambda_{\alpha}^2 + \frac{C}{\delta_2} (2 n - 2 k + 1) f_{\beta} (D_{\tau_{\beta}} W)^2.
\end{aligned} \]
Hence, \eqref{eq2B-15} reduces to
\[  - \mathcal{G}^i D_i W + \mathcal{G}^i u_n \rho_{t i}
\leq  C  |D W| + 2 \delta_2  \sum_{\alpha \neq \beta} f_{\alpha} \lambda_{\alpha}^2 + \frac{C}{\delta_2} \Big( \sum f_{\alpha} (D_{\tau_{\alpha}} W)^2 + \sum f_{\alpha} \Big).
\]
Taking this inequality, \eqref{eq2B-11}, \eqref{eq2B-3} into \eqref{eq2B-8}, and choosing $\delta_2 < \frac{\epsilon_2^2}{8 \sup u}$, we obtain
\begin{equation} \label{eq2B-13}
 L W \leq C |D W| + C \Big( \sum f_{\alpha} (D_{\tau_{\alpha}} W)^2 + \sum f_{\alpha} \Big).
\end{equation}

{\bf Subcase (ii).}  Suppose $2 \sigma_{k - 1} (\lambda | \beta) \leq \sigma_{k - 1}$.
Then we have $\lambda_{\beta} > 0$.

If $\sigma_{k} (\lambda | \beta) \geq 0$, then
\[ 0 < f_{\beta} \lambda_{\beta} = \frac{1}{k} \sigma_{k}^{\frac{1}{k} - 1} \Big( \sigma_{k} - \sigma_{k} (\lambda | \beta) \Big) \leq  \frac{1}{k} \sigma_{k}^{\frac{1}{k}}. \]
Consequently,
\begin{equation} \label{eq2B-12}
\Big\vert \frac{2 P_{\beta q} u_{q} }{u w} f_{\beta} \lambda_{\beta}  D_{\tau_{\beta}} W \Big\vert \leq C |D W|.
\end{equation}

Now we assume $\sigma_{k} (\lambda | \beta) < 0$. By \eqref{eq3.21} and $\partial_i = \frac{1}{u} P_{\alpha l} \gamma_{li} \tau_{\alpha}$,
\[\begin{aligned}
& D_{\tau_{\beta}} W =  u_{t \tau_{\beta}} + \rho_{t \tau_{\beta}} u_n + \rho_t u_{n \tau_{\beta}} - \sum_{s < n} u_s u_{s \tau_{\beta}} \\
= & (P_{\beta l} \gamma_{lt} + P_{\beta l} \gamma_{ln} \rho_t) (\lambda_{\beta} - 1) + \rho_{t \tau_{\beta}} u_n -  \sum_{s < n} u_s P_{\beta l} \gamma_{l s} (\lambda_{\beta} - 1).
\end{aligned}\]
It follows that,
\begin{equation} \label{eq2B-16}
\begin{aligned}
& \frac{2 P_{\beta q} u_{q} }{u w} f_{\beta} \lambda_{\beta}  D_{\tau_{\beta}} W =  \frac{1}{k} \sigma_k^{\frac{1}{k} - 1} \Big(\sigma_k - \sigma_{k}(\lambda | \beta) \Big) \frac{2 P_{\beta q} u_{q} }{u w}  D_{\tau_{\beta}} W \\
 \leq & C |D W| -  \frac{1}{k} \sigma_k^{\frac{1}{k} - 1} \sigma_{k}(\lambda | \beta) \frac{2 P_{\beta q} u_{q} }{u w}  D_{\tau_{\beta}} W \\
 \leq & C |D W| -  \frac{1}{k} \sigma_k^{\frac{1}{k} - 1} \sigma_{k}(\lambda | \beta) \Big( C ( \epsilon_1 + |\rho_t| ) \lambda_{\beta} + C \Big).
\end{aligned}
\end{equation}
Note that
\begin{equation} \label{eq2B-17}
\begin{aligned}
- \frac{1}{k} \sigma_k^{\frac{1}{k} - 1} \sigma_{k}(\lambda | \beta) =  - \frac{1}{k} \sigma_k^{\frac{1}{k} - 1} \Big(\sigma_k - \lambda_{\beta} \sigma_{k - 1} (\lambda | \beta)\Big)
=   \frac{k - 1}{k} \sigma_k^{\frac{1}{k}} - \sum_{\alpha \neq \beta} f_{\alpha} \lambda_{\alpha}.
\end{aligned}
\end{equation}
Also, using an inequality of Ivochkina \cite{Ivochkina89} (see also an improved version of Lin-Trudinger \cite{Lin-Trudinger94-1})
\[ \sigma_{k + 1} (\lambda | \beta) \leq C(n, k) \sum_{\alpha \neq \beta} \sigma_{k - 1} (\lambda | \alpha) \lambda_{\alpha}^2, \]
we have
\begin{equation} \label{eq2B-18}
\begin{aligned}
& -  \frac{1}{k} \sigma_k^{\frac{1}{k} - 1} \sigma_{k}(\lambda | \beta) \lambda_{\beta} =
-  \frac{1}{k} \sigma_k^{\frac{1}{k} - 1}  \Big( \sigma_{k + 1} - \sigma_{k + 1} (\lambda | \beta) \Big) \\
= &    \frac{1}{k} \sigma_k^{\frac{1}{k} - 1} \Big( \frac{1}{k} \sum_{\alpha \neq \beta} \sigma_{k - 1} (\lambda | \alpha) \lambda_{\alpha}^2 + \frac{1 + k}{k} \sigma_{k + 1} (\lambda | \beta) - \frac{1}{k} \sigma_k \sigma_1(\lambda | \beta) \Big) \\
\leq &  \frac{1}{k} \sigma_k^{\frac{1}{k} - 1} \Big(  C(n, k) \sum_{\alpha \neq \beta} \sigma_{k - 1} (\lambda | \alpha) \lambda_{\alpha}^2 - \frac{1}{k} \sigma_k \sigma_1(\lambda | \beta) \Big)
\leq  C \sum_{\alpha \neq \beta} f_{\alpha} \lambda_{\alpha}^2  + C,
\end{aligned}
\end{equation}
where the last inequality is true because if $k \geq 2$, then $\sigma_1(\lambda | \beta) > 0$; while if $k = 1$,
\[ C(n, k) \sum_{\alpha \neq \beta} \sigma_{k - 1} (\lambda | \alpha) \lambda_{\alpha}^2 - \frac{1}{k} \sigma_k \sigma_1(\lambda | \beta) \leq C \sum_{\alpha \neq \beta} \lambda_{\alpha}^2 + C. \]
By \eqref{eq2B-17} and \eqref{eq2B-18}, inequality \eqref{eq2B-16} becomes
\begin{equation} \label{eq2B-19}
 \frac{2 P_{\beta q} u_{q} }{u w} f_{\beta} \lambda_{\beta}  D_{\tau_{\beta}} W \leq C |D W| + \Big( \frac{\delta_2}{4} + C ( \epsilon_1 + |\rho_t| ) \Big) \sum\limits_{\alpha \neq \beta} f_{\alpha} \lambda_{\alpha}^2 + \frac{C}{\delta_2} \sum f_{\alpha}.
\end{equation}
Taking \eqref{eq2B-19} (which covers the case \eqref{eq2B-12}) into \eqref{eq2B-15},  then taking the resulting inequality as well as \eqref{eq2B-11}, \eqref{eq2B-3} into \eqref{eq2B-8}, and choosing $\epsilon_1$, $r$ further small depending on $\delta_2$, $\delta_2 < \frac{\epsilon_2^2}{8 \sup u}$, we obtain
\begin{equation} \label{eq2B-14}
L W \leq C \Big( |D W| + \sum f_{\alpha} + \sum f_{\alpha} (D_{\tau_{\alpha}} W)^2 \Big).
\end{equation}
Note that \eqref{eq2B-14} covers the cases \eqref{eq2B-10} and \eqref{eq2B-13}.

Now, take
\[ V = 1 - e^{- a W} - b |x|^2. \]
By direct calculation, Lemma \ref{Lemma2B} and \eqref{eq2B-14}, we can verify that over $\Omega_{\epsilon} \cap B_r (0)$,
\[ \begin{aligned}
L V \leq &  C \Big( |D V| + 2 b r \Big) + a e^{- a W} C \Big( \sum f_{\alpha} + \sum f_{\alpha} (D_{\tau_{\alpha}} W)^2 \Big) \\
& - a^2 e^{- a W}  \frac{1}{u} \sum f_{\alpha} (D_{\tau_{\alpha}} W)^2 - 2 b  \frac{u}{w^2} \sum f_{\alpha} + C b r.
\end{aligned} \]
Choosing $a$ large, then $b$ large, and $r$ small, we have
\begin{equation} \label{eq2B-20}
L V \leq  C |D V|.
\end{equation}

Now, we only need the following linear operator
\[ \mathcal{L}  =  \mathcal{G}^{ij} D_{ij}. \]
By \eqref{eq2B-20}, we have on $\Omega_{\epsilon} \cap B_r (0)$,
\begin{equation} \label{eq2B-21}
\mathcal{L} V  \leq  C |D V|.
\end{equation}

\vspace{2mm}

\subsection{Barrier construction}

Let $d(x)$ be the distance from $x$ to $\Gamma_{\epsilon}$ in $\mathbb{R}^n$. Consider the barrier as in \cite{Lin-Trudinger94},
\[ B(x) = - a_0 |x|^2 + c_0 ( e^{- b_0 d(x)} - 1 ), \]
where $a_0$, $b_0$ and $c_0$ are positive constants to be determined.
By assumption \eqref{eq1-3}, the principal curvatures of $\Gamma_{\epsilon}$ with respect to $\gamma$ satisfy
\[ (\kappa'_1, \ldots, \kappa'_{n - 1}) \in \Gamma'_k \quad \mbox{on} \,\, \Gamma_{\epsilon}. \]
Choose $r$ sufficiently small such that $d$ is $C^4$ within $\{ x \in \overline{\Omega_{\epsilon}} | d(x) \leq r \}$ and
\[ \Big( \frac{\kappa'_1}{1 - \kappa'_1 d}, \cdots, \frac{\kappa'_{n - 1}}{1 - \kappa'_{n - 1} d} \Big) \in \Gamma'_k. \]
Choose $a_0$ sufficiently large (depending on $r$) such that
\begin{equation} \label{eq2B-24}
B \leq V  \quad \mbox{on} \quad \partial(\Omega_{\epsilon} \cap B_r (0)).
\end{equation}

For fixed $x \in \Omega_{\epsilon} \cap B_r (0)$, let $d(x) = |x - y|$ with $y \in \Gamma_{\epsilon}$. We shall use the principal coordinate system at $y$. Denote $\kappa'_1, \ldots, \kappa'_{n - 1}$ the principal curvatures of $\Gamma_{\epsilon}$ at $y$. Then we have
\[ D^2 B =  - 2 a_0 I + c_0 b_0 e^{- b_0 d}  \mbox{diag} \Big( \frac{\kappa'_1}{1 - \kappa'_1 d}, \cdots, \frac{\kappa'_{n - 1}}{1 - \kappa'_{n - 1} d}, b_0 \Big). \]

By concavity of $\mathcal{G} (r, p, z)$ with respect to $r$,
\begin{equation} \label{eq2B-22}
\begin{aligned}
& \mathcal{L} B - C |D B| = \mathcal{G}^{ij} (D_{ij} B - d_0 \delta_{ij} ) + d_0 \sum \mathcal{G}^{ii} - C |D B|  \\
\geq \,& \mathcal{G}(D^2 B - d_0 I, D u, u) - \mathcal{G}(D^2 u, D u, u) + \mathcal{G}^{ij} D_{ij} u + d_0 \sum \mathcal{G}^{ii} - C |D B| \\
\geq \, & \mathcal{G}(D^2 B - d_0 I, D u, u) - \sum f_{\alpha} + \frac{d_0 u}{w^2} \sum f_{\alpha}  - C \Big(2 a_0 r + c_0 b_0 e^{- b_0 d} \Big) \\
\geq \, &  \mathcal{G}(D^2 B - d_0 I, D u, u)  - C c_0 b_0 e^{- b_0 d},
\end{aligned}
\end{equation}
where the last inequality is true when constant $d_0$ is sufficiently large.

Note that if
\begin{equation} \label{eq2B-25}
\lambda\Big( u ( B_{ij} - d_0 \delta_{ij}) + \delta_{ij} + u_i u_j \Big) \in \Gamma_{k + 1},
\end{equation}
then
\[ \begin{aligned}
 \mathcal{G}(D^2 B - d_0 I, D u, u) = & F \Big( \gamma^{\alpha i} \big( u ( B_{ij} - d_0 \delta_{ij}) + \delta_{ij} + u_i u_j  \big) \gamma^{j \beta} \Big) \\
 \geq & \frac{1}{(1 + |D u|^2)^{1/k}} F \Big(  u ( B_{ij} - d_0 \delta_{ij}) + \delta_{ij} + u_i u_j  \Big).
\end{aligned} \]
Take this inequality into \eqref{eq2B-22},
\begin{equation} \label{eq2B-26}
 \mathcal{L} B - C |D B|
 \geq   c_1 F \Big(  u ( B_{ij} - d_0 \delta_{ij}) + \delta_{ij} + u_i u_j  \Big) - C c_0 b_0 e^{- b_0 d},
\end{equation}
where $c_1$ is a fixed positive constant.

Choose $b_0$ sufficiently large such that
\[ \Lambda := \mbox{diag} \Big( \frac{\kappa'_1}{1 - \kappa'_1 d}, \cdots, \frac{\kappa'_{n - 1}}{1 - \kappa'_{n - 1} d}, b_0 \Big) \in \Gamma_{k + 1} \quad \mbox{and} \quad    c_1 F (  u \Lambda ) > C. \]
Then choose $c_0$ sufficiently large such that
\[ \lambda\Big( - (2 a_0 + d_0) I + c_0 b_0 e^{- b_0 d} \Lambda \Big) \in \Gamma_{k + 1} \]
and
\[  c_1 F \Big(  - \frac{u (2 a_0 + d_0) e^{ b_0 d}}{c_0 b_0} I + u \Lambda  \Big) > C. \]
Therefore, \eqref{eq2B-25} is true and \eqref{eq2B-26} reduces to
\begin{equation} \label{eq2B-27}
 \mathcal{L} B  \geq  C |D B|.
\end{equation}
By \eqref{eq2B-21}, \eqref{eq2B-27}, \eqref{eq2B-24}, the maximum principle and $V(0) = B(0)$, we obtain $u_{tn} (0) \geq - \frac{c_0 b_0}{a}$. If we replace $W$ by $- u_t - u_n \rho_{t} - \frac{1}{2} \sum_{s < n} u_s^2$, by the same argument, we will obtain $u_{tn} (0) \leq \frac{c_0 b_0}{a}$.

\vspace{2mm}

\subsection{Double normal derivative estimate}~

\vspace{1mm}

We shall give an upper bound for $D_{\gamma\gamma} u$ on $\Gamma_{\epsilon}$.
For $x \in \Gamma_{\epsilon}$, define
\[ \tilde{d} (x)  = w \, \mbox{dist} ( \kappa'(x), \partial \Gamma'_{k - 1} ), \]
where $\kappa' = (\kappa'_1, \ldots, \kappa'_{n-1})$ are the roots of
\[ \det ( \kappa'_{\zeta}  g_{\alpha\beta} - h_{\alpha \beta}  ) = 0, \]
and $(g_{\alpha \beta})$, $(h_{\alpha \beta})$ are the first $(n - 1) \times (n - 1)$ principal minors of $(g_{ij})$ and $(h_{ij})$ with the indices $\alpha, \beta < n$ running over the tangential directions on $\Gamma_{\epsilon}$ and $n$ indicates the normal direction to $\Gamma_{\epsilon}$. Throughout this subsection, the range for Greek letter $\alpha, \beta, \ldots$ is from $1$ to $n - 1$.

Here in this subsection, $\kappa'$ is different from the one defined in the introduction. Note that $\kappa' \in \Gamma'_{k - 1}$ since $\kappa \in \Gamma_k$.
Assume the minimum of $\tilde{d}(x)$ along $\Gamma_{\epsilon}$ is achieved at $0 \in \Gamma_{\epsilon}$, at which we fix the coordinate system with the positive $x_n$ axis points to the interior normal of $\Gamma_{\epsilon}$ at $0$. We want to prove that $\tilde{d}(0)$ has a uniform positive lower bound.

Choose a local orthonormal frame $e_1, \ldots, e_{n}$ around $0$ on $\Omega_{\epsilon}$, obtained by parallel translation of a local orthonormal frame $e_1, \ldots, e_{n - 1}$ around $0$ on $\Gamma_{\epsilon}$ satisfying
\[ (h_{\alpha \beta})(0)  \text{ is diagonal with }   h_{11}(0) \leq \ldots \leq h_{n - 1, n - 1}(0), \]
and $e_n = \gamma$ along the lines perpendicular to $\Gamma_{\epsilon}$ on $\Omega_{\epsilon}$.
In what follows in this subsection, we may simply write a Greek letter $\alpha$ instead of $e_{\alpha}$ in the subscripts with $\alpha < n$; while use a Latin letter $s$ in the subscripts to represent $\partial_s$. We can check that the local frame
\[ \xi_1 = \epsilon e_1, \ldots, \xi_{n-1} = \epsilon e_{n - 1}\]
around $0$ on $\Gamma_{\epsilon}$ satisfies
\[ g_{\xi_\alpha \xi_\beta} = \delta_{\alpha\beta}, \quad  h_{\xi_\alpha \xi_\beta}(0) = \kappa'_{\alpha}(0) \delta_{\alpha\beta}, \quad \kappa'_1 (0) \leq \ldots \leq \kappa'_{n-1} (0). \]

By Lemma 6.1 of \cite{CNSIII}, there exists $\mu' = (\mu_1, \ldots, \mu_{n-1}) \in \mathbb{R}^{n - 1}$ with
\[ \mu_1 \geq \ldots \geq \mu_{n - 1} \geq 0 \text{ and }  \sum \mu_{\alpha}^2 = 1 \]
such that
$ \Gamma'_{k-1} \subset \{ \kappa' \in \mathbb{R}^{n-1} \,|\, \mu' \cdot \kappa' > 0 \}$ and
\[  \tilde{d}(0) = w \sum \mu_{\alpha} \kappa'_{\alpha} (0) =   \sum \mu_{\alpha} \big( 1 + u u_{\alpha \alpha} \big) (0). \]
We may assume $\tilde{d} (0) \leq  \frac{1}{2}$, for otherwise we are done. Note that $u_{\alpha\beta}  = u_{\gamma} d_{\alpha \beta}$ and $u_{\gamma} \geq \underline{u}_{\gamma} > 0$ on $\Gamma_{\epsilon}$. Hence we obtain
\[ \sum \mu_{\alpha} d_{\alpha \alpha} (0) \leq - c_2
\]
for some positive constant $c_2$.
By continuity of $d_{\alpha \alpha}$ at $0$,
\[    \sum \mu_{\alpha} \,d_{\alpha\alpha} (x) \leq - \frac{c_2}{2}  \quad \quad\mbox{in} \quad
 \Omega_\epsilon \cap B_{r}(0)  \]
for some positive constant $r$.
Also, by Lemma 6.2 of \cite{CNSIII}, for any $x \in \Gamma_{\epsilon}$ near $0$,
\[  \sum \mu_{\alpha} \big( 1 + u u_{\gamma} d_{\alpha \alpha} \big)  =  \sum \mu_{\alpha} \big( 1 + u u_{\alpha \alpha} \big)
\geq w \sum \mu_{\alpha} \kappa'_{\alpha} (x) \geq \tilde{d} (x) \geq \tilde{d} (0). \]
Thus, we can define in $\Omega_\epsilon \cap B_{r}(0)$,
\[ \Phi  =   \frac{1}{\epsilon \sum \mu_{\alpha} d_{\alpha \alpha} } \Big( \tilde{d}(0) - \sum \mu_{\alpha} \Big) - D_{e_n} u - \frac{K}{2} \sum_{s < n} u_s^2. \]
Obviously,  $\Phi +  \frac{K}{2} \sum_{s < n} u_s^2 \geq 0$ on $\Gamma_{\epsilon} \cap B_{r} (0)$ and $\Phi(0) = 0$.
In addition, similar as how we derive \eqref{eq2B-14}, by choosing $K$ sufficiently large we have in $\Omega_\epsilon \cap B_{r}(0)$,
\[ L(\Phi) \leq   C \Big( |D \Phi| + \sum f_{\alpha} + \sum f_{\alpha} (D_{\tau_{\alpha}} \Phi)^2 \Big).  \]
Taking $V = 1 - e^{- a \Phi} - b |x|^2$, and choosing $a$ sufficiently large, then $b$ sufficiently large, we can verify that over $\Omega_{\epsilon} \cap B_r (0)$ for sufficiently small $r$,
$L V \leq  C |D V|$.
Thus, on $\Omega_{\epsilon} \cap B_r (0)$,
$\mathcal{L} V  \leq  C |D V|$.
By the maximum principle, we have $B_n(0) \leq V_n(0)$. Therefore, $u_{nn} (0) \leq C$ and
$|D^2 u (0)| \leq C$. Consequently, we obtain a bound for all principal curvatures of graph of $u$ at $0$. Since $\psi > 0$ on $\Gamma_{\epsilon}$,
$\mbox{dist} ( \kappa(0), \partial\Gamma_k)$ has a uniform positive lower bound. Consequently,
$\tilde{d}(0)$ has a uniform positive lower bound.  By applying Lemma 1.2 of \cite{CNSIII} and similar to the proof in \cite{Sui2019}, we proved
$u_{\gamma\gamma} \leq C$ on $\Gamma_{\epsilon}$.

\vspace{4mm}

\section{The approximating Dirichlet problem \eqref{eqn10}}

\vspace{4mm}

In this section, we write equation \eqref{eq1-1} as
\begin{equation} \label{eqn17}
G( D^2 u, D u, u ) = F ( a_{ij} ) =  f( \lambda ( a_{ij} ) ) = \psi( x, u ).
\end{equation}

\vspace{1mm}

\subsection{Existence}
Motivated by Su \cite{Su16}, we construct a two-step continuity process to prove the existence.
For convenience, denote
\[ G[u] = \, G (D^2 u, D u, u), \quad  G^{ij}[u] = G^{ij} (D^2 u, D u, u), \quad \mbox{etc.}\]
Let $\delta$ be a small positive constant such that
\begin{equation} \label{eq3-14}
G[\underline{u}] =  G( D^2 \underline{u}, D \underline{u}, \underline{u} ) >  \delta \underline{u} \quad\mbox{in}\,\, \Omega_{\epsilon}.
\end{equation}
For $t \in [0, 1]$, consider the following two equations.
\begin{equation} \label{eq3-12}
\left\{ \begin{aligned}
G (D^2 u, D u, u)  =  &  \Big(  ( 1 - t ) \frac{G[\underline{u}](x)}{ \underline{u} } + t \delta \Big)  u \quad & \mbox{in} \,\, \Omega_{\epsilon}, \\
u   = &  \epsilon \quad & \mbox{on} \,\, \Gamma_{\epsilon}.
\end{aligned} \right.
\end{equation}
\begin{equation} \label{eq3-13}
\left\{ \begin{aligned}
G (D^2 u, D u, u)  =  &  ( 1 - t ) \delta u  +  t  \psi(x, u) \quad & \mbox{in} \,\, \Omega_{\epsilon}, \\
u   = &  \epsilon  \quad & \mbox{on} \,\, \Gamma_{\epsilon}.
\end{aligned} \right.
\end{equation}

\begin{lemma} \label{Lemma6-1}
For $x \in \overline{\Omega_{\epsilon}}$ and a positive $C^2$ function $u$ which is admissible near $x$, if
\[G [u] (x) =  F ( a_{ij}[u] )(x) = f (\kappa)(x) = \psi(x) u, \]
then we have
\[ G_u [u] (x) - \psi(x) < 0. \]
\end{lemma}
\begin{proof}
\begin{equation*}
G_u  = F^{ij} \frac{1}{w} \gamma^{ik} u_{k l} \gamma^{lj} = \frac{1}{u} \Big( \sum f_i \kappa_i - \frac{1}{w} \sum f_i \Big).
\end{equation*}
Since $f$ is homogeneous of degree one, thus $\sum f_i \kappa_i = \psi(x) u$. Consequently,
\[ G_u [ u ] (x) - \psi(x)  =   - \frac{1}{w u} \sum f_i  < 0. \]
\end{proof}

\begin{lemma}  \label{Lemma6-2}
For $t \in [0, 1]$,  let $\underline{U}$ and $u$ be any admissible subsolution and solution of \eqref{eq3-12}. Then $u \geq \underline{U}$ in $\Omega_{\epsilon}$. In particular, \eqref{eq3-12} has at most one admissible solution.
\end{lemma}
\begin{proof}
If not, $\underline{U} - u$ achieves a positive maximum at $x_0 \in \Omega_{\epsilon}$, and
\begin{equation} \label{eq3-15}
\underline{U}(x_0) > u(x_0),\quad D \underline{U}(x_0) = D u(x_0), \quad D^2\underline{U}(x_0) \leq D^2 u(x_0).
\end{equation}
Note that for any $s \in [0, 1]$, the deformation $u[s] = s \underline{U} + (1 - s) u$ is admissible near $x_0$. This is because at $x_0$,
\[ \begin{aligned}
   & \delta_{ij} + u[s]  {\gamma}^{ik}  ( u[s] )_{kl}  \gamma^{lj}
  \geq  \delta_{ij} + u[s] {\gamma}^{ik}  \underline{U}_{kl} \gamma^{lj}  \\
   = & (1 - s) \Big( 1 - \frac{u}{\underline{U}}\Big) \delta_{ij} + \frac{u[s]}{\underline{U}} \Big( \delta_{ij} + \underline{U}  \gamma^{ik} \underline{U}_{kl} \gamma^{lj} \Big).
  \end{aligned}
\]
For $s \in [0, 1]$, define a differentiable function
\[ a(s) = G \Big[ u[s] \Big] (x_0) - \Big(  ( 1 - t ) \frac{G [\underline{u}] (x_0)}{ \underline{u}(x_0) } + t \delta \Big)  u[s](x_0). \]
Since $a(0) = 0$ and $a(1) \geq 0$, there exists $s_0 \in [0, 1]$ such that $a(s_0) = 0$ and $a'(s_0) \geq 0$, that is,
\begin{equation} \label{eq3-16}
 G\big[ u[s_0] \big] (x_0) = \Big(  ( 1 - t ) \frac{G [\underline{u}] (x_0)}{ \underline{u}(x_0) } + t \delta \Big) u[s_0] (x_0),
\end{equation}
and
\begin{equation} \label{eq3-17}
\begin{aligned}
& G^{ij}\big[ u[s_0]  \big](x_0)  D_{ij}  (\underline{U} - u)(x_0)
 + G^i \big[ u[s_0]  \big](x_0)  D_i  (\underline{U} - u)(x_0)
 \\ & +  \Big(G_u \big[ u[s_0] \big](x_0) - \big ( ( 1 - t ) \frac{G [\underline{u}] (x_0)}{ \underline{u}(x_0) } + t \delta \big) \Big)  (\underline{U} - u)(x_0) \geq 0.
\end{aligned}
\end{equation}
However, inequality \eqref{eq3-17} can not hold by \eqref{eq3-15}, \eqref{eq3-16} and Lemma \ref{Lemma6-1}.
\end{proof}

\begin{thm} \label{Theorem6-1}
For $t \in [0, 1]$, \eqref{eq3-12} has a unique admissible solution $u \geq \underline{u}$.
\end{thm}
\begin{proof}
Uniqueness is proved in Lemma \ref{Lemma6-2}. We use standard continuity method to prove the existence. By \eqref{eq3-14},  $\underline{u}$ is a subsolution of \eqref{eq3-12}. The $C^2$ estimate for admissible solution $u \geq \underline{u}$ of \eqref{eq3-12} implies uniform ellipticity of this equation, which further gives $C^{2, \alpha}$ estimate by Evans-Krylov theory
\begin{equation} \label{eq3-20}
\Vert u \Vert_{C^{2, \alpha} ( \overline{ \Omega_{\epsilon} } )}  \leq C,
\end{equation}
where $C$ is independent of $t$.
Denote
\[ C_0^{2, \alpha} ( \overline{ \Omega_{\epsilon} } ) = \{ w \in C^{2, \alpha}(  \overline{ \Omega_{\epsilon} }  ) \,| \,w = 0 \,\, \mbox{on} \,\, \Gamma_{\epsilon} \}, \]

\[ \mathcal{U} = \{ w \in C_0^{2, \alpha} ( \overline{ \Omega_{\epsilon} } ) \,| \, \underline{u} + w \,\,\mbox{is}\,\,\mbox{admissible} \,\, \mbox{in} \,\, \overline{\Omega_{\epsilon}}  \}. \]
Obviously, $C_0^{2, \alpha} ( \overline{ \Omega_{\epsilon} } )$ is a subspace of $C^{2, \alpha}( \overline{ \Omega_{\epsilon} } )$ and
$\mathcal{U}$ is an open subset of $C_0^{2, \alpha} (\overline{ \Omega_{\epsilon} })$.
Define $\mathcal{L}: \mathcal{U} \times [ 0, 1 ] \rightarrow C^{\alpha}( \overline{ \Omega_{\epsilon} } )$,
\[ \mathcal{L} ( w, t ) = G [ \underline{u} + w ]  -  \Big(  ( 1 - t ) \frac{G[\underline{u}]}{ \underline{u} } + t \delta  \Big)  (\underline{u} + w),  \]
and set
\[ \mathcal{S} = \{ t \in [0, 1] \,|\, \mathcal{L}(w, t) = 0 \,\,\mbox{has}\,\,\mbox{a}\,\,\mbox{solution}\,\,w \,\,\mbox{in}\,\,\mathcal{U} \}. \]
Since $\mathcal{L}(0, 0) = 0$, $\mathcal{S} \neq \emptyset$.

$\mathcal{S}$ is open in $[0, 1]$. In fact, for any $t_0 \in \mathcal{S}$, there exists $w_0 \in \mathcal{U}$ such that $\mathcal{L} (w_0, t_0) = 0$. Note that the Fr\'echet derivative of $\mathcal{L}$ with respect to $w$ at $(w_0, t_0)$ is a linear elliptic operator from $C^{2, \alpha}_0 ( \overline{\Omega_{\epsilon}} )$ to $C^{\alpha}( \overline{\Omega_{\epsilon}})$,
\[ \begin{aligned}
\mathcal{L}_w \big|_{(w_0, t_0)} ( h )  =    G^{ij}[\underline{u} + w_0] D_{ij} h  +  G^i [ \underline{u} + w_0 ] D_i h  \\
+ \Big(G_u [ \underline{u} + w_0] - ( 1 - t_0 ) \frac{G[\underline{u}]}{ \underline{u} } - t_0 \delta   \Big) h.
\end{aligned} \]
Lemma \ref{Lemma6-1} implies $\mathcal{L}_w \big|_{(w_0, t_0)}$ is invertible. Thus a neighborhood of $t_0$ is also contained in $\mathcal{S}$ by implicit function theorem.

$\mathcal{S}$ is closed in $[0, 1]$. In fact, let $t_i$ be a sequence in $\mathcal{S}$ converging to $t_0 \in [0, 1]$ and $w_i \in \mathcal{U}$ be the unique (by Lemma \ref{Lemma6-2}) solution to $\mathcal{L} (w_i, t_i) = 0$. Lemma \ref{Lemma6-2} implies $w_i \geq 0$, and \eqref{eq3-20} implies that $u_i = \underline{u} + w_i$ is a bounded sequence in $C^{2, \alpha}(\overline{\Omega_{\epsilon}})$, which possesses a subsequence converging to an admissible solution $u_0$ of \eqref{eq3-12}. Since $w_0 = u_0 - \underline{u} \in \mathcal{U}$ and $\mathcal{L}(w_0, t_0) = 0$, we know that $t_0 \in \mathcal{S}$.
\end{proof}

Now we may assume $\underline{u}$ is not a solution of \eqref{eqn10}, for otherwise we are done.

\begin{lemma} \label{Lemma6-3}
If $u \geq \underline{u}$ is an admissible solution of \eqref{eq3-13}, then
$u > \underline{u}$ in $\Omega_{\epsilon}$ and $(u - \underline{u})_{\gamma} > 0$ on $\Gamma_{\epsilon}$.
\end{lemma}

\begin{proof}
Indeed, we can write \eqref{eq3-13} in a more general form.
\begin{equation}  \label{eq3-18}
\left\{ \begin{aligned}
F( A[u] ) = &  \psi(x, u)  \quad  &\mbox{in} \,\, & \Omega_{\epsilon}, \\
u = &  \varphi  \quad  & \mbox{on} \,\, & \Gamma_{\epsilon}. \end{aligned} \right.
\end{equation}
Since $\underline{u}$ is a subsolution but not a solution of \eqref{eq3-18}, we have
\[ F( A[\underline{u}]) - F(A [u]) \geq \psi (x, \underline{u}) - \psi (x, u). \]
Also,
\[\begin{aligned}
F( A [\underline{u}]) - F( A [u]) =  \int_0^1 \frac{d}{d s} F ( (1 - s) A [u] + s A[ \underline{u} ] ) d s \\
=  (a_{ij}[\underline{u}] - a_{ij}[u])  \int_0^1 F^{ij} ( (1 - s) A [u] + s A[ \underline{u} ] ) d s
\end{aligned} \]
and
\[ \begin{aligned}
& a_{ij}[\underline{u}] - a_{ij}[u] = a_{ij}(D^2 \underline{u}, D \underline{u}, \underline{u}) - a_{ij}(D^2 u, D u, u) \\
= & a_{ij}(D^2 \underline{u}, D \underline{u}, \underline{u}) - a_{ij}(D^2 \underline{u}, D \underline{u}, u) + a_{ij}(D^2 \underline{u}, D \underline{u}, u) - a_{ij}(D^2 \underline{u}, D u, u) \\
& + a_{ij}(D^2 \underline{u}, D u, u) - a_{ij}(D^2 u, D u, u).
\end{aligned} \]
Applying the Maximum Principle and Lemma H (see p. 212 of \cite{GNN}) we proved the lemma.
\end{proof}

\begin{thm} \label{Theorem6-2}
For any $t \in [0, 1]$, there is an admissible solution $u \geq \underline{u}$ to Dirichlet problem \eqref{eq3-13}.
\end{thm}

\begin{proof}
By classical Schauder theory, the $C^{2, \alpha}$ estimate for admissible solution $u \geq \underline{u}$ of \eqref{eq3-13} further implies $C^{4, \alpha}$ estimate
\begin{equation} \label{eq3-21}
\Vert u \Vert_{C^{4,\alpha}(\overline{\Omega_{\epsilon}})} < C_4.
\end{equation}
In addition,
\begin{equation} \label{eq3-22}
\mbox{dist}(\kappa[u], \partial\Gamma_k) > c_2 > 0 \quad \mbox{in} \,\, \overline{\Omega_{\epsilon}},
\end{equation}
where $C_4$ and $c_2$ are independent of $t$. Denote
\[ C_0^{4, \alpha} (\overline{\Omega_{\epsilon}}) = \{ w \in C^{4, \alpha}( \overline{\Omega_{\epsilon}}) \,|\, w = 0 \,\, \mbox{on} \,\, \Gamma_{\epsilon} \}\]
and
\[ \mathcal{O} = \Big\{ w \in C_0^{4, \alpha} (\overline{\Omega_{\epsilon}}) \left\vert \begin{footnotesize}\begin{aligned} & w > 0 \,\,\mbox{in}\,\,\Omega_{\epsilon}, \,\,  w_{\gamma} > 0 \,\,\mbox{on}\,\, \Gamma_{\epsilon},  \,\, \Vert w {\Vert}_{C^{4,\alpha}(\overline{\Omega_{\epsilon}})} < C_4 + \Vert\underline{u}\Vert_{C^{4,\alpha}(\overline{\Omega_{\epsilon}})}\\ & \underline{u} + w  \,\, \mbox{is} \,\, \mbox{admissible} \,\, \mbox{in} \,\, \overline{\Omega_{\epsilon}}, \,\, \mbox{dist}(\kappa[\underline{u} + w], \partial\Gamma_k) > c_2  \,\, \mbox{in} \,\, \overline{\Omega_{\epsilon}}  \end{aligned}\end{footnotesize} \right.\Big\}. \]
We know that $\mathcal{O}$ is a bounded open subset of $C_0^{4, \alpha} (\overline{\Omega_{\epsilon}})$.

Define a map
$\mathcal{M}_t (w):  \mathcal{O} \times [ 0, 1 ] \rightarrow C^{2,\alpha}(\overline{\Omega_{\epsilon}})$,
\[ \mathcal{M}_t (w) = G [ \underline{u} + w ]  - ( 1 - t ) \delta (\underline{u} + w)  -  t  \psi(x, \underline{u} + w). \]
By Theorem \ref{Theorem6-1} and Lemma \ref{Lemma6-2}, there is a unique admissible solution $u^0$ of \eqref{eq3-12} at $t = 1$, which is also the unique admissible solution of \eqref{eq3-13} for $t = 0$. By Lemma \ref{Lemma6-2}, $w^0 = u^0 - \underline{u} \geq 0$ in $\Omega_{\epsilon}$. Consequently, $w^0 > 0$ in $\Omega_{\epsilon}$ and ${w^0}_{\gamma} > 0$ on $\Gamma_{\epsilon}$ by Lemma \ref{Lemma6-3}. Meanwhile, $\underline{u} + w^0$ satisfies \eqref{eq3-21} and \eqref{eq3-22}. Thus, $w^0 \in \mathcal{O}$. In view of Lemma \ref{Lemma6-3}, \eqref{eq3-21} and \eqref{eq3-22}, $\mathcal{M}_t(w) = 0$ has no solution on $\partial\mathcal{O}$ for any $t \in [0, 1]$. Note that $\mathcal{M}_t$ is uniformly elliptic on $\mathcal{O}$ independent of $t$. Hence we can define the degree of $\mathcal{M}_t$ on $\mathcal{O}$ at $0$, which is independent of $t$. It suffices to show this degree is nonzero at $t = 0$.
We have known that $\mathcal{M}_0 ( w ) = 0$ has a unique solution $w^0 \in \mathcal{O}$. The Fr\'echet derivative of $\mathcal{M}_0$ with respect to $w$ at $w^0$ is a linear elliptic operator from $C^{4, \alpha}_0 (\overline{\Omega_{\epsilon}})$ to $C^{2, \alpha}(\overline{\Omega_{\epsilon}})$,
\begin{equation*}
\mathcal{M}_{0, w} |_{w^0} ( h )  =  G^{ij}[ u^0 ] D_{ij} h  + G^i [ u^0 ] D_i h   + ( G_u [ u^0 ]  - \delta ) h.
\end{equation*}
By Lemma \ref{Lemma6-1}, $G_u [ u^0 ]  - \delta  < 0$ in $\overline{\Omega_{\epsilon}}$. Hence $\mathcal{M}_{0, w} |_{w^0}$ is invertible. By degree theory in \cite{Li89} we can conclude that the degree at $t = 0$ is nonzero, which implies that \eqref{eq3-13} has at least one admissible solution $u \geq \underline{u}$ for any $t \in [0, 1]$.
\end{proof}

\vspace{2mm}

\subsection{Comparison principle, uniqueness and monotonicity}~

\vspace{2mm}

We have the following comparison principle.
\begin{thm} \label{TheoremComparison}
Under assumption \eqref{eq1-2}, let $\underline{U}$ and $u$ be any admissible subsolution and solution of
\eqref{eqn17} in $\Omega_{\epsilon}$
and $u \geq \underline{U}$ on $\Gamma_{\epsilon}$.
Then $u \geq \underline{U}$ in $\Omega_{\epsilon}$.
\end{thm}
\begin{proof}
If not, $\underline{U} - u$ achieves a positive maximum at $x_0 \in \Omega_{\epsilon}$, at which,
\begin{equation} \label{eq5-5}
\underline{U}(x_0) > u(x_0),\quad D \underline{U}(x_0) = D u(x_0), \quad D^2\underline{U}(x_0) \leq D^2 u(x_0).
\end{equation}
Note that for any $s \in [0, 1]$, the deformation $u[s] = s \underline{U} + (1 - s) u$ is admissible near $x_0$.
For $s \in [0, 1]$, define a differentiable function
\[ a(s) = G \big[ u[s] \big] (x_0) - \psi\big(x_0, u[s] \big). \]
Since $a(0) = 0$ and $a(1) \geq 0$, there exists $s_0 \in [0, 1]$ such that $a(s_0) = 0$ and $a'(s_0) \geq 0$, that is,
\begin{equation} \label{eq5-6}
 G\big[ u[s_0] \big] (x_0) = \psi\big(x_0, u[s_0] \big),
\end{equation}
and
\begin{equation} \label{eq5-7}
\begin{aligned}
& G^{ij}\big[ u[s_0]  \big](x_0)  D_{ij}  (\underline{U} - u)(x_0)
 + G^i \big[ u[s_0]  \big](x_0)  D_i  (\underline{U} - u)(x_0)
 \\ & +  \Big(G_u \big[ u[s_0] \big](x_0) - \psi_u (x_0, u[s_0]) \Big)  (\underline{U} - u)(x_0) \geq 0.
\end{aligned}
\end{equation}
However, inequality \eqref{eq5-7} can not hold by \eqref{eq5-5}, \eqref{eq5-6} and the fact that
\[
G_u \big[ u[s_0] \big](x_0) - \psi_u (x_0, u[s_0]) = \frac{1}{u[s_0]} \Big( \psi\big(x_0, u[s_0] \big) - \frac{1}{w} \sum f_i \Big) - \psi_u \big(x_0, u[s_0] \big) < 0.
\]
\end{proof}
By Theorem \ref{TheoremComparison}, we obtain the uniqueness part of Theorem \ref{Theorem1}.  Besides, we can deduce the following monotonicity property of $u^{\epsilon}$ with respect to $\epsilon$.
\begin{cor}
Under the assumptions of Theorem \ref{Theorem1}, for $0 < \epsilon_1 < \epsilon_2$, we have $u^{\epsilon_1} \geq u^{\epsilon_2}$ in $\Omega_{\epsilon_2}$.
\end{cor}

\vspace{4mm}

\section{Interior estimates}

\vspace{2mm}
\,

\subsection{\large Interior gradient estimate}~

\vspace{2mm}

Let $u^{\epsilon} \geq \underline{u}$ be an admissible solution over $\Omega_{\epsilon}$ to the Dirichlet problem \eqref{eqn10}. For any fixed $\epsilon_0 > 0$, we want to establish the uniform $C^1$ estimate for $u^{\epsilon}$ for any $0 < \epsilon < \frac{\epsilon_0}{2}$ on $\overline{\Omega_{\epsilon_0}}$,  namely,
\begin{equation} \label{eq4-1}
\Vert u^{\epsilon} \Vert_{C^1(\overline{\Omega_{\epsilon_0}})} \leq C, \quad \forall \,\, 0 < \epsilon < \frac{\epsilon_0}{2}.
\end{equation}
Hereinafter, $C$ represents a positive constant which is independent of $\epsilon$, but may depend on $\epsilon_0$.

By Lemma \ref{lemma5-1}, we obtain uniform $C^0$ estimate:
\[ u^{\epsilon} \leq C \quad \mbox{on} \,\, \overline{\Omega_{\epsilon}}, \quad \forall\, \epsilon > 0. \]
In particular, we have
\begin{equation} \label{eq4-2}
\frac{\epsilon_0}{2} \leq  u^{\epsilon}  \leq  C \quad \mbox{on} \quad \overline{\Omega_{\epsilon_0/2}}, \quad\forall \,\, 0 < \epsilon < \frac{\epsilon_0}{2}.
\end{equation}
Choose $r = \mbox{dist}(\overline{\Omega_{\epsilon_0}}, \Gamma_{\epsilon_0 / 2})$, and cover $\overline{\Omega_{\epsilon_0}}$ by finitely many open balls $B_{\frac{r}{2}}$ with radius $\frac{r}{2}$ and centered in $\Omega_{\epsilon_0}$. Note that the number of such open balls depends on $\epsilon_0$. In addition, the corresponding balls $B_r$ are all contained in $\Omega_{\epsilon_0 / 2}$, over which, we are able to apply \eqref{eq4-2}.
Now we want to establish interior gradient estimate on each $B_r$ by applying Wang's idea \cite{WangXJ98}.
Since the gradient $D u^{\epsilon}$ are invariant under change of Euclidean coordinate system, we may assume the center of $B_r$ is $0$. For convenience, we also omit the superscript in $u^{\epsilon}$ and write as $u$.

For $x \in B_r(0)$ and $\xi \in \mathbb{S}^{n-1}$, consider the test function
\begin{equation*}
\Theta (x, u, \xi) =  \,  \ln \rho(x) + \varphi(u) +  \ln \ln  u_{\xi},
\end{equation*}
where $\rho(x) = (r^2 - |x|^2)^2$ with $|x|^2 = \sum_{i = 1}^n x_i^2$ and $\varphi(u) = \ln u$.

By the definition of the test function, we know that the maximum value of $\Theta$ must be attained in an interior point $x^0 = (x_1, \ldots, x_n) \in B_r(0)$. We choose the Euclidean coordinate frame ${\partial}_1, \ldots, {\partial}_n$ around $x^0$ such that the direction obtaining the maximum is $\xi = {\partial}_1$. Then at $x^0$,
\[ u_1 = |D u| \quad \mbox{and} \quad u_i = 0 \quad \mbox{for} \quad i = 2, \ldots, n. \]
Therefore, \eqref{eq3.4} holds.
Rotate ${\partial}_2, \ldots, {\partial}_n$ such that at $x^0$, $\big\{ u_{\alpha \beta} \big\}_{\alpha, \beta \geq 2}$ is diagonal and $u_{22} \geq \ldots \geq u_{nn}$. Consequently, we have
\begin{equation} \label{eq3.5}
 a_{ij} = \,\frac{1}{w} \big( \delta_{ij} + u \gamma^{ik} u_{kl} \gamma^{lj} \big) = \left\{ \begin{aligned}
 & \frac{1}{w} \Big( 1 + \frac{u u_{11}}{w^2} \Big), \quad \mbox{if} \quad i = j = 1, \\
 & \frac{u u_{ij}}{w^2}, \quad \mbox{if} \quad i = 1 \,\,\mbox{or} \,\, j = 1, \,\, \mbox{and} \,\, i + j > 2, \\
 &  \frac{1}{w} \big( 1 + u u_{ii} \big) \delta_{ij}, \quad \mbox{otherwise}.
  \end{aligned} \right.
\end{equation}

Since the function
\[    \ln \rho(x) + \varphi(u) +  \ln \ln  u_1    \]
achieves its maximum at $x^0$, we have at $x^0$,
\begin{equation} \label{eq3.1}
\frac{{\rho}_i}{\rho} + \varphi'(u) u_i + \frac{u_{1i}}{u_1 \ln u_1} = 0,
\end{equation}

\vspace{2mm}

\begin{equation} \label{eq3.2}
\begin{aligned}
& \frac{G^{ij} \rho_{ij}}{\rho} - \frac{G^{ij} \rho_i \rho_j}{\rho^2}  +  \varphi'(u) G^{ij} u_{ij} + \varphi''(u) G^{ij} u_i u_j \\
& + \frac{G^{ij} u_{1ij}}{u_1 \ln u_1} - \frac{\ln u_1 + 1}{(u_1 \ln u_1)^2} G^{ij} u_{1i} u_{1j} \leq 0.
\end{aligned}
\end{equation}

By Lemma \ref{Lemma1}, \eqref{eq3.4}, \eqref{eq3.5}, we can compute
\begin{equation}  \label{eq3.6}
 - G^s u_{s1} - G_u u_1
=   F^{ij} b_{ij} + \Big( \frac{u_1 u_{11}}{w^2} - \frac{u_1}{u} \Big) \psi  + \frac{u_1}{u w} \sum f_i,
\end{equation}
where
\begin{equation}   \label{eq3.7}
b_{ij} = b_{ji} = \left\{ \begin{aligned}
&  \frac{2 u u_1}{w^5} u_{11}^2 + \frac{2 u u_1}{w^3 (1 + w)} \sum_{k > 1} u_{1 k}^2, \quad i = j = 1, \\
& \frac{u u_1 (1 + 2 w)}{w^4 (1 + w)} u_{11} u_{1j} + \frac{u u_1}{w^2 (1 + w)} u_{1j} u_{jj}, \quad i = 1,\,\, j > 1,  \\
& \frac{2 u u_1}{w^2 (1 + w)} u_{1i} u_{1j},  \quad i,\,\, j > 1.
\end{aligned} \right.
\end{equation}

Combining \eqref{eq2.4}, \eqref{eq3.6} and \eqref{eq3.7} yields,
\begin{equation} \label{eq3.8}
\begin{aligned}
& \frac{G^{ij} u_{1ij}}{u_1 \ln u_1} - \frac{\ln u_1 + 1}{(u_1 \ln u_1)^2} G^{ij} u_{1i} u_{1j}  \\
\geq\, & F^{11} \Big( \frac{2 u}{w^5 \ln u_1} - \frac{u}{w^3} \frac{\ln u_1 + 1}{(u_1 \ln u_1)^2} \Big) u_{11}^2 + \frac{2 u}{w^2 (1 + w) \ln u_1} \sum\limits_{j > 1} F^{1j} u_{1j} u_{jj}\\
& + \sum\limits_{j > 1} F^{1j} u_{11} u_{1j} \frac{2 u}{w^2 (w + 1) \ln u_1} \Big( \frac{1 + 2 w}{w^2} - \frac{\ln u_1 + 1}{(w - 1) \ln u_1} \Big) \\
& + \frac{1}{u w \ln u_1} \sum f_i + \frac{\psi_{x_1} + \psi_u u_1}{u_1 \ln u_1} + \Big( \frac{u_1 u_{11}}{w^2} - \frac{u_1}{u} \Big) \frac{\psi}{u_1 \ln u_1}
\end{aligned}
\end{equation}
when $u_1$ is sufficiently large.

From \eqref{eq3.1}, we have
\begin{equation}  \label{eq3.16}
\frac{u_{11}}{u_1 \ln u_1} = - \frac{\rho_1}{\rho} - \varphi'(u) u_1.
\end{equation}
We may assume $\vert\frac{\rho_1}{\rho}\vert \leq \frac{1}{2} \varphi'(u) u_1$, for otherwise, we are done. Then
\begin{equation} \label{eq3.10}
u_{11} \leq - \frac{1}{2} \varphi' u_1^2 \ln u_1 < 0.
\end{equation}

Also, note that for $j = 2, \ldots, n$,
\[ F^{1j} = - \frac{1}{k} \sigma_k^{\frac{1}{k} - 1} a_{j1} \sigma_{k - 2} ( a_{22}, \ldots, a_{nn} \vert a_{jj} ). \]
Therefore, in view of \eqref{eq3.5},
\begin{equation} \label{eq3.12}
F^{1j} u_{1j} \leq 0, \quad \quad j = 2, \ldots, n.
\end{equation}

Denote $J = \{ 2 \leq j \leq n \,\vert \,  u_{jj} \geq 0 \}$.
By \eqref{eq3.10} and \eqref{eq3.12}, when $u_1$ is sufficiently large, \eqref{eq3.8} reduces to
\begin{equation} \label{eq3.9}
\begin{aligned}
& \frac{G^{ij} u_{1ij}}{u_1 \ln u_1} - \frac{\ln u_1 + 1}{(u_1 \ln u_1)^2} G^{ij} u_{1i} u_{1j}  \\
\geq \, & \frac{u}{2 w^5 \ln u_1} F^{11} u_{11}^2 + \frac{2 u}{w^2 (1 + w) \ln u_1} \sum\limits_{j \in J} F^{1j} u_{1j} u_{jj} \\
& + \frac{1}{u w \ln u_1} \sum f_i + \frac{\psi_{x_1} + \psi_u u_1}{u_1 \ln u_1} + \Big( \frac{u_1 u_{11}}{w^2} - \frac{u_1}{u} \Big) \frac{\psi}{u_1 \ln u_1}.
\end{aligned}
\end{equation}

By \eqref{eq3.5} and \eqref{eq3.10}, we further obtain
\[ a_{11} =  \frac{1}{w} \Big( 1 + \frac{u u_{11}}{w^2} \Big) \leq  \frac{1}{w} \Big( 1 - \frac{u  \varphi' u_1^2 \ln u_1}{2 w^2} \Big) < 0  \]
as $u_1$ is sufficiently large. It follows that
\begin{equation} \label{eq3.13}
\begin{aligned}
& F^{11} =    \frac{1}{k} \sigma_k^{\frac{1}{k} - 1}  \sigma_{k-1}(a_{22}, \ldots, a_{nn})  \\
=  & \frac{1}{k} \sigma_k^{\frac{1}{k} - 1} \Big(\sigma_{k-1} + \sum_{j = 2}^n a_{1j}^2 \sigma_{k - 3}(a_{22}, \ldots, a_{nn} | a_{jj})
 - a_{11} \sigma_{k-2}(a_{22}, \ldots, a_{nn}) \Big)  \\
\geq &   \frac{1}{k} \sigma_k^{\frac{1}{k} - 1} \sigma_{k-1}.
\end{aligned}
\end{equation}

For $j \in J$, by \eqref{eq3.5},
\begin{equation} \label{eq3.14}
\begin{aligned}
& F^{1j} u_{1j} u_{jj} =  - \frac{1}{k} \sigma_k^{\frac{1}{k} - 1} a_{1j} \sigma_{k - 2} ( a_{22}, \ldots, a_{nn} \vert a_{jj} ) u_{1j} u_{jj}  \\
= &  - \frac{1}{k} \sigma_k^{\frac{1}{k} - 1} \frac{u u_{1j}}{w^2} \sigma_{k - 2} ( a_{22}, \ldots, a_{nn} \vert a_{jj} ) u_{1j} \frac{w a_{jj} - 1}{u} \\
\geq &  - \sigma_k^{\frac{1}{k} - 1} \frac{u_{1j}^2}{w} C(n, k) a_{22} \cdots a_{kk}
+ \frac{1}{k} \sigma_k^{\frac{1}{k} - 1} \frac{u_{1j}^2}{w^2} \sigma_{k - 2} ( a_{22}, \ldots, a_{nn} \vert a_{jj} )  \\
\geq & - \frac{u_{1j}^2}{w}  C(n, k) F^{11},
\end{aligned}
\end{equation}
where in the last line, we have applied
$\sigma_{k - 1} (a_{22}, \ldots, a_{nn}) \geq  a_{22} \cdots a_{kk}$ (see formula (19) in \cite{Lin-Trudinger94-1}).

Also by \eqref{eq3.1}, we have
\begin{equation} \label{eq3.11}
u_{1j} = - u_1 \ln u_1 \frac{\rho_j}{\rho}, \quad \quad j = 2, \ldots, n.
\end{equation}

By \eqref{eq3.10}, \eqref{eq3.14} and \eqref{eq3.11},  the inequality \eqref{eq3.9} reduces to
\begin{equation} \label{eq3.15}
\begin{aligned}
& \frac{G^{ij} u_{1ij}}{u_1 \ln u_1} - \frac{\ln u_1 + 1}{(u_1 \ln u_1)^2} G^{ij} u_{1i} u_{1j}  \\
\geq \, & \frac{u}{8 w^5} F^{11} \varphi'^2 u_1^4  \ln u_1 - \frac{C(n, k) u |D \rho|^2 u_1^2 \ln u_1}{\rho^2 w^3 (1 + w)} F^{11} \\
& + \frac{1}{u w \ln u_1} \sum f_i + \frac{\psi_{x_1} + \psi_u u_1}{u_1 \ln u_1} + \Big( \frac{u_1 u_{11}}{w^2} - \frac{u_1}{u} \Big) \frac{\psi}{u_1 \ln u_1}.
\end{aligned}
\end{equation}
We may assume that
\[  \frac{\varphi'^2 u_1^2}{16 w^2}    \geq \frac{C(n, k)  |D \rho|^2}{\rho^2  (1 + w)},   \]
for otherwise we are done.  Also in view of \eqref{eq3.16}, inequality \eqref{eq3.15} further reduces to
\begin{equation} \label{eq3.17}
\begin{aligned}
& \frac{G^{ij} u_{1ij}}{u_1 \ln u_1} - \frac{\ln u_1 + 1}{(u_1 \ln u_1)^2} G^{ij} u_{1i} u_{1j}  \\
\geq \, & \frac{u_1^4 \ln u_1}{16 u w^5} F^{11}
+ \frac{\sum f_i}{u w \ln u_1} + \frac{\psi_{x_1} + \psi_u u_1}{u_1 \ln u_1} - \frac{\rho_1 u_1 \psi}{\rho w^2} - \frac{u_1^2 \psi}{u w^2} - \frac{\psi}{u \ln u_1}.
\end{aligned}
\end{equation}

For the rest terms in \eqref{eq3.2}, by Lemma \ref{Lemma1} and \eqref{eq3.4} we have
\begin{equation} \label{eq3.18}
\begin{aligned}
& \frac{G^{ij} \rho_{ij}}{\rho} - \frac{G^{ij} \rho_i \rho_j}{\rho^2} = G^{ij} \Big( - \frac{4 \delta_{ij} (r^2 - |x|^2)}{\rho} - \frac{8 x_i x_j}{\rho} \Big) \\ \geq & - \frac{8 r^2 \sum G^{ii}}{\rho} =  - \frac{8 r^2 u}{\rho w} \Big( \frac{1}{w^2} F^{11} + \sum\limits_{i > 1} F^{ii} \Big) \geq  - \frac{8 r^2 u}{\rho w} \sum F^{ii},
\end{aligned}
\end{equation}
and
\begin{equation} \label{eq3.19}
\begin{aligned}
\varphi'(u) G^{ij} u_{ij} + \varphi''(u) G^{ij} u_i u_j
= \, & \varphi' \big( \psi - \frac{1}{w} \sum F^{ii} \big) + \varphi'' \frac{u}{w^3} F^{11} u_1^2 \\
= \, & \frac{1}{u} \big( \psi - \frac{1}{w} \sum F^{ii} \big) - \frac{1}{u w^3} F^{11} u_1^2.
\end{aligned}
\end{equation}
Taking \eqref{eq3.17}--\eqref{eq3.19} into \eqref{eq3.2} yields,
\begin{equation} \label{eq3.20}
\begin{aligned}
& \Big( \frac{u_1^4 \ln u_1}{16 u w^5} - \frac{u_1^2}{u w^3} \Big) F^{11} - \Big( \frac{1}{u w} + \frac{8 r^2 u}{\rho w} - \frac{1}{u w \ln u_1} \Big) \sum F^{ii} \\
& - \frac{\rho_1 u_1 \psi}{\rho w^2} - \frac{C}{u_1 \ln u_1} + \frac{\psi}{u w^2} + \Big( \psi_u - \frac{\psi}{u} \Big) \frac{1}{\ln u_1} \leq 0.
\end{aligned}
\end{equation}
By \eqref{eq3.13} and Newton-Maclaurin inequality,
\[ c(n, k) \leq \sum F^{ii} = \frac{n - k + 1}{k} \sigma_k^{\frac{1}{k} - 1} \sigma_{k - 1} \leq (n - k + 1) F^{11}, \]
where $c(n, k)$ is a positive constant. Therefore by assumption \eqref{eq1-2}, we can deduce $\rho \ln u_1 \leq C$ from \eqref{eq3.20}.

\begin{rem}
In \cite{Weng2019}, Weng also derived the interior gradient estimate. Our test function is slightly different from Weng and the resulting estimate depends on $n$, $k$, $r$, $\Vert u \Vert_{C^0( B_r )}$ and $\Vert \psi \Vert_{C^1(B_r)}$. Our calculation may be easier.
\end{rem}

\vspace{2mm}

\subsection{A remark on second order interior estimates}~

\vspace{2mm}

In \cite{Sui2019}, we generalized Guan-Qiu's interior curvature estimate for convex solutions to prescribed scalar curvature equations to hyperbolic space. However, for $k \geq 3$, there is no such estimate. Hence it is natural to think of Pogorelov type interior curvature estimate. In this subsection, we first formulate some possible domains on which we wish to establish Pogorelov interior curvature estimate, but then we observe an obstruction.

For $0 < \epsilon \leq \epsilon_0$, define
\[ \Omega^{\epsilon}_{\epsilon_0} := \{ x\in \Omega_{\epsilon} \,\big\vert\, u^{\epsilon}(x) > \epsilon_0 \}. \]
It is easy to check the following properties of $\Omega^{\epsilon}_{\epsilon_0}$.

\vspace{2mm}

\begin{prop}~ \label{Prop1}
Under the assumptions of Theorem \ref{Theorem1},
\begin{enumerate} [(a)]
  \item For $0 < \epsilon < \epsilon_1 < \epsilon_2$, we have $\Omega_{\epsilon_2}^{\epsilon} \subset \Omega_{\epsilon_1}^{\epsilon}$;
  \item For $0 < \epsilon_1 < \epsilon_2 < \epsilon_0$, we have $\Omega_{\epsilon_0}^{\epsilon_2} \subset \Omega_{\epsilon_0}^{\epsilon_1}$;
  \item For any $\epsilon > 0$, we have $\Omega_{\epsilon}^{\epsilon} = \Omega_{\epsilon}$;
  \item For any $0 < \epsilon < \epsilon_0$, we have $\Omega_{\epsilon_0} \subset \Omega_{\epsilon_0}^{\epsilon} \subset \Omega_{\epsilon}$;
  \item For any $0 < \epsilon < \epsilon_0$, we have $u^{\epsilon} = \epsilon_0$ on $\partial \Omega_{\epsilon_0}^{\epsilon}$.
\end{enumerate}
\end{prop}

In order to find a domain containing all $\Omega_{\epsilon_0}^{\epsilon}$ for sufficiently small $\epsilon$ which also stays away from $\Gamma$, we want to find a supersolution to the asymptotic Plateau problem \eqref{eqn9} and utilize its level set.

Note that by Newton-Maclaurin inequality,
\[  \sigma_1 (\kappa [\underline{u}]) \geq \sigma_k^{1/k} ( \kappa [\underline{u}] ) \geq \psi(x, \underline{u}) \quad \mbox{in} \,\, \Omega. \]
Thus, $\underline{u}$ is a subsolution to the mean curvature equation
\begin{equation} \label{eq7-2}
\left\{ \begin{aligned}
\sigma_1 ( \kappa [ u ] ) =  & \, \psi(x, u) \quad & \mbox{in} \,\, \Omega, \\
u  = & \,  0 \quad & \mbox{on} \,\, \Gamma.
\end{aligned} \right.
\end{equation}
By the estimates in the previous sections, we can find a unique smooth solution $\overline{u} \geq \underline{u}$ to \eqref{eq7-2}. Again by Newton-Maclaurin inequality, we have
\[ \sigma_k^{1/k} ( \kappa [\overline{u}] ) \leq \sigma_1 ( \kappa [ \overline{u} ] ) =   \psi(x, \overline{u}) \]
or $\kappa [ \overline{u} ] \notin \Gamma_k$,
which means $\overline{u}$ is a smooth supersolution to \eqref{eqn9}.

Now for $\epsilon_0 > 0$, define
\[  \hat{\Omega}_{\epsilon_0} := \{ x \in \Omega \,\big\vert\, \overline{u}(x) \geq \epsilon_0 \}.  \]

\vspace{2mm}

\begin{prop} \label{Prop2}~
Under the assumptions of Theorem \ref{Theorem1}, denote
$\delta_{\epsilon_0} = \min\limits_{\hat{\Omega}_{\epsilon_0}} \underline{u}$.
\begin{enumerate} [(a)]
  \item $0 < \delta_{\epsilon_0} \leq \epsilon_0$;
  \item For $0 < \epsilon < \delta_{\epsilon_0}$, we have
$\Omega_{\epsilon_0}^{\epsilon} \subset \hat{\Omega}_{\epsilon_0} \subset \overline{\Omega_{\delta_{\epsilon_0}}}$.
\end{enumerate}
\end{prop}
By Proposition \ref{Prop1}, \ref{Prop2} and estimate \eqref{eq4-1}, we have
\[ \overline{\Omega_{\epsilon_0}} \subset \Omega_{\epsilon_0 / 2} \subset \Omega_{\epsilon_0 / 2}^{\epsilon} \subset  \overline{\Omega_{\delta_{\epsilon_0/2}}} \]
and
\[ \Vert u^{\epsilon} \Vert_{C^1(\overline{\Omega_{\delta_{\epsilon_0/2}}})} \leq C, \quad \forall \,\, 0 < \epsilon < \frac{1}{2} \delta_{\epsilon_0/2}. \]
Thus we wish to establish Pogorelov type interior curvature estimate
\begin{equation} \label{eq7-5}
\big| \kappa_i [ u^{\epsilon}] (x) \big| \,\leq \, \frac{C}{(u^{\epsilon} - \frac{\epsilon_0}{2})^b}, \quad \forall \,\, x \in \Omega_{\epsilon_0 / 2}^{\epsilon}, \quad \forall \,\, 0 < \epsilon < \frac{1}{2} {\delta}_{\epsilon_0 / 2}
\end{equation}
for some positive constants $b$ and $C$ independent of $\epsilon$ (may depend on $\epsilon_0$), because then
we would obtain uniform $C^2$ bound
\[ \max\limits_{\overline{\Omega_{\epsilon_0}}} \big| \kappa_i [ u^{\epsilon}] \big|\,\leq \, C, \quad \forall \,\, 0 < \epsilon <  \frac{1}{2} {\delta}_{\epsilon_0 / 2}, \]
which would imply the existence of a smooth solution to asymptotic Plateau problem \eqref{eqn9}.

However, it is impossible to establish interior Pogorelov type estimate \eqref{eq7-5}.
In fact, by \eqref{eqC-3}, \eqref{eq0-2}, \eqref{eq2G-13} and \eqref{eq0-3}, we have
\[ \begin{aligned}
 b \frac{F^{ii} \nabla_{ii} u}{u - \frac{\epsilon_0}{2}} - b F^{ii} \frac{u_i^2}{( u - \frac{\epsilon_0}{2} )^2}
 = \frac{b u (u - \epsilon_0)}{ (u - \frac{\epsilon_0}{2})^2} \sum f_i \frac{u_i^2}{u^2} + \frac{b \psi \nu^{n+1} u}{u - \frac{\epsilon_0}{2}} - \frac{b  u }{u - \frac{\epsilon_0}{2}} \sum F^{ii}.
\end{aligned} \]
Because of the term $- \frac{b  u }{u - \frac{\epsilon_0}{2}} \sum F^{ii}$, we are unable to use Sheng-Urbas-Wang's method \cite{Sheng-Urbas-Wang} to establish the estimate \eqref{eq7-5}. This term comes out due to the ambient space $\mathbb{H}^{n+1}$.

\vspace{4mm}

\section{Viscosity solutions}

\vspace{4mm}

In this section, we verify that $u$ in Theorem \ref{Theorem2} is indeed a viscosity solution of
\begin{equation} \label{eq8-1}
G( D^2 u, D u, u ) = F ( a_{ij} ) =  f( \lambda ( a_{ij} ) ) = \psi( x, u ).
\end{equation}
We first give the definition of viscosity solutions of \eqref{eq8-1}, according to the definitions given by Trudinger \cite{Trudinger90} and Urbas \cite{Ur90} in Euclidean space.
\begin{defn}
A function $0 < u \in C^0(\Omega)$ is a viscosity subsolution of \eqref{eq8-1} in $\Omega$ if for any function $\phi \in C^2(\Omega)$, any $x_0 \in \Omega$ satisfying $u(x_0) = \phi(x_0)$ and $u \leq \phi$ in a neighborhood $\Omega_{x_0} \subset \Omega$ of $x_0$, we have $G(D^2 \phi, D \phi, \phi)(x_0) \geq \psi(x_0, \phi(x_0))$.
A function $0 < u \in C^0(\Omega)$ is a viscosity supersolution of \eqref{eq8-1} in $\Omega$ if for any function $\phi \in C^2(\Omega)$, any $x_0 \in \Omega$ satisfying $u(x_0) = \phi(x_0)$ and $u \geq \phi$ in a neighborhood $\Omega_{x_0} \subset \Omega$ of $x_0$, we have either $\phi$ is not admissible at $x_0$, or $G(D^2 \phi, D \phi, \phi)(x_0) \leq \psi(x_0, \phi(x_0))$.
A function $u$ is a viscosity solution of \eqref{eq8-1} if it is both a viscosity subsolution and supersolution.
\end{defn}

By this definition, we can verify the following fact.
\begin{prop} \label{Prop3}
A function $0 < u \in C^2(\Omega)$ is a viscosity solution of \eqref{eq8-1} if and only if it is an admissible classical solution.
\end{prop}
\begin{proof}
First, let $0< u \in C^2(\Omega)$ be a viscosity solution of \eqref{eq8-1}. We claim that $u$ is admissible in $\Omega$. Suppose not, say $u$ is not admissible at some $x_0 \in \Omega$. There exists a unique $\alpha_0 \geq 0$ such that
\begin{equation} \label{eq8-2}
\lambda\Big(  \frac{1}{w} \big( \delta_{ij} + u \gamma^{ik} (u_{kl} + \alpha_0 \delta_{kl}) \gamma^{lj} \big) \Big)(x_0) \in \partial\Gamma_k
\end{equation}
and
\[ \lambda\Big(  \frac{1}{w} \big( \delta_{ij} + u \gamma^{ik} (u_{kl} + \alpha \delta_{kl}) \gamma^{lj} \big) \Big)(x_0) \in \Gamma_k, \quad \forall \,\,\alpha > \alpha_0. \]
For any $\alpha > \alpha_0$, consider the function
\[ \phi(x) = u(x_0) + D u(x_0) \cdot (x - x_0) + \frac{1}{2}(x - x_0) (D^2 u(x_0) + \alpha I) (x - x_0)^{T}. \]
It is easy to verify that $\phi$ is admissible at $x_0$, $\phi(x_0) = u(x_0)$ and $\phi \geq u$ in a neighborhood of $x_0$.
Since $u$ is a viscosity subsolution of \eqref{eq8-1}, we have
\[  G(D^2 u + \alpha I, D u, u)(x_0) = G(D^2 \phi, D \phi, \phi)(x_0) \geq \psi(x_0, \phi) = \psi(x_0, u). \]
However, as $\alpha \rightarrow \alpha_0$,
\[ G \big(D^2 u + \alpha_0 I, D u, u \big)(x_0) \geq \psi \big(x_0, u(x_0)\big) > 0, \]
contradicting \eqref{eq8-2}. Hence $u$ is admissible in $\Omega$. By definition of viscosity solution, taking $\phi = u$, we can verify that $u$ is a classical solution.

The converse direction can be easily proved by definition of viscosity solution.
\end{proof}

Now we prove a kind of stability result.
\begin{prop}
The solution $u$ in Theorem \ref{Theorem2} is a viscosity solution of \eqref{eq8-1} in $\Omega$.
\end{prop}
\begin{proof}
The proof follows the idea of Lions \cite{Lions83}. First, we show that $u$ is a viscosity subsolution of \eqref{eq8-1} in $\Omega$. For any $\phi \in C^2(\Omega)$ and any $x_0 \in \Omega$ satisfying $u(x_0) = \phi(x_0)$ and $u < \phi$ in a neighborhood $\Omega_{x_0} \setminus \{ x_0 \}$, let $\delta > 0$ be sufficiently small such that $\overline{B_{\delta}(x_0)} \subset \Omega_{x_0}$, where $B_{\delta}(x_0)$ is an open ball centered at $x_0$ with radius $\delta$. Then
\[ \max\limits_{\partial B_{\delta}(x_0)} (u - \phi) < 0.  \]
Since $u^{\epsilon}$ locally uniformly converges to $u$ in $\Omega$, we have
\[  \max\limits_{\overline{B_{\delta}(x_0)}} (u^{\epsilon} - \phi) >  \max\limits_{\partial B_{\delta}(x_0)} (u^{\epsilon} - \phi) \quad  \mbox{with} \quad \overline{B_{\delta}(x_0)} \subset \Omega_{\epsilon}  \]
as $\epsilon$ sufficiently small. We may in addition choose $\epsilon = \epsilon(\delta)$ in such a way that $\epsilon \rightarrow 0^+$ as $\delta \rightarrow 0^+$.
Therefore, there exists $x_{\delta} \in B_{\delta} (x_0)$ such that
\[  \max\limits_{\overline{B_{\delta}(x_0)}} (u^{\epsilon} - \phi) = (u^{\epsilon} - \phi)(x_{\delta}). \]
Since $u^{\epsilon}$ is a classical admissible solution of \eqref{eq8-1} in $\Omega_{\epsilon}$, by Proposition \ref{Prop3}, it is certainly a viscosity solution of \eqref{eq8-1} in $\Omega_{\epsilon}$. Hence
\[ \phi_{\delta} = \phi + \max\limits_{\overline{B_{\delta}(x_0)}} (u^{\epsilon} - \phi) \]
satisfies
\[ G\big(D^2 \phi_{\delta}, D \phi_{\delta}, \phi_{\delta} \big)(x_{\delta}) \geq \psi\big(x_{\delta}, \phi_{\delta}(x_{\delta})\big). \]
Letting $\delta \rightarrow 0^+$, we have $x_{\delta} \rightarrow x_0$ and $u^{\epsilon}(x_{\delta}) \rightarrow u(x_0)$, or equivalently, $\phi_{\delta} (x_{\delta}) \rightarrow \phi(x_0)$. Since $\phi$ is $C^2$, we have $D \phi_{\delta} (x_{\delta}) = D \phi (x_{\delta}) \rightarrow D \phi(x_0)$ and $D^2 \phi_{\delta} (x_{\delta}) = D^2 \phi (x_{\delta}) \rightarrow D^2 \phi (x_0)$ as $\delta \rightarrow 0^+$. Consequently,
\[ G\big(D^2 \phi, D \phi, \phi \big)(x_0) \geq \psi\big(x_0, \phi(x_0)\big). \]
This implies that $u$ is a viscosity subsolution of \eqref{eq8-1} in $\Omega$. Similarly, we can verify that $u$ is a viscosity supersolution.
\end{proof}

\end{document}